\documentclass[12pt,reqno]{amsart}
\usepackage{a4wide}
\usepackage{microtype}
\usepackage[all]{xy}
\usepackage{amsmath}
\usepackage{amsfonts}
\usepackage{amssymb}
\usepackage{amsthm}
\usepackage{amscd}
\usepackage{here}
\usepackage[
bookmarks=false,
colorlinks=true,
debug=true,
naturalnames=true,
pdfnewwindow=true,
citecolor=blue,
linkcolor=blue,
urlcolor = blue]{hyperref}
\usepackage[alphabetic,msc-links,initials]{amsrefs}
\usepackage{verbatim}
\usepackage{graphicx,xcolor}
\usepackage{mathrsfs}
\usepackage{upgreek}
\usepackage{tikz}
\usepackage[colorinlistoftodos]{todonotes}
\usepackage{mathtools}
\newcommand{\seq}{\coloneqq}
\usepackage{multirow}

\numberwithin{equation}{section}

\newtheorem{Thm}{Theorem}[section]
\newtheorem{Cor}[Thm]{Corollary}
\newtheorem{Prop}[Thm]{Proposition}
\newtheorem{Lem}[Thm]{Lemma}
\newtheorem{Conj}[Thm]{Conjecture}
\newtheorem{THM}{Theorem}

\newtheorem{CONJ}[THM]{Conjecture}

\theoremstyle{definition}
\newtheorem{Def}[Thm]{Definition}
\newtheorem{Rem}[Thm]{Remark}
\newtheorem{Ex}[Thm]{Example}

\newcommand{\ol}{\overline}

\newcommand{\Hom}{\mathop{\mathrm{Hom}}\nolimits}

\newcommand{\Aut}{\mathop{\mathrm{Aut}}\nolimits}

\newcommand{\Ext}{\mathop{\mathrm{Ext}}\nolimits}
\newcommand{\Ker}{\mathop{\mathrm{Ker}}\nolimits}
\newcommand{\Cok}{\mathop{\mathrm{Cok}}\nolimits}

\newcommand{\id}{\mathrm{id}}

\newcommand{\Z}{\mathbb{Z}}
\newcommand{\Q}{\mathbb{Q}}
\newcommand{\C}{\mathbb{C}}

\newcommand{\kk}{\Bbbk}

\newcommand{\fg}{\mathfrak{g}}
\newcommand{\fh}{\mathfrak{h}}

\newcommand{\op}{\mathrm{op}}
\newcommand{\sQ}{\mathsf{Q}}

\newcommand{\tC}{\widetilde{C}}
\newcommand{\tc}{\widetilde{c}}

\newcommand{\gmod}{\text{-$\mathrm{gmod}$}}
\newcommand{\umod}{\text{-$\mathrm{mod}$}}
\newcommand{\ghom}{\mathop{\mathrm{hom}}\nolimits}
\newcommand{\gext}{\mathop{\mathrm{ext}}\nolimits}
\newcommand{\tQ}{\widetilde{Q}}
\newcommand{\ep}{\varepsilon}
\newcommand{\Cc}{\mathcal{C}}
\newcommand{\CC}{\mathscr{C}}

\newcommand{\fac}{\mathop{\mathrm{fac}}\nolimits}
\newcommand{\sub}{\mathop{\mathrm{sub}}\nolimits}
\newcommand{\tPi}{\widetilde{\Pi}}
\newcommand{\bD}{\mathbb{D}}

\newcommand{\Gr}{\mathrm{Gr}}

\newcommand{\hq}{\mathfrak{h}^*_{q,t}}

\newcommand{\eg}{\textrm{e.g}.~}

\newcommand{\lf}{\mathrm{l.f.}}

\newcommand{\se}{\mathrm{s}}
\newcommand{\tl}{\mathrm{t}}
\newcommand{\kb}{\mathbf{k}}
\newcommand{\Div}{\mathop{\mathrm{Div}}\nolimits}
\newcommand{\zero}{\mathop{\mathrm{zero}}\nolimits}
\newcommand{\Gg}{\mathcal{G}}

\subjclass[2020]{16G20, 17B37, 16W50, 17B67, 81R50}
\keywords{Deformed Cartan matrices; Generalized preprojective algebras; $(q, t)$-analogues and bigraded modules; Quantum affine algebras; $R$-matrices}
\thanks{R.F. was supported by JSPS Overseas Research Fellowships.}
\thanks{K.M. was supported by the Kyoto Top Global University project, Grant-in-Aid for JSPS Fellows (JSPS KAKENHI Grant Number JP21J14653) and JSPS bilateral program (Grant
Number JPJSBP120213210). This work was partly supported by Osaka City University Advanced Mathematical Institute: MEXT Joint Usage/Research Center on Mathematics and Theoretical Physics JPMXP0619217849.}

\title[DCM \and GPA]
{Deformed Cartan matrices and generalized preprojective algebras I: Finite type}

\date{\today}

\author[R.~Fujita]{Ryo Fujita}
\address[R.~Fujita]{Research Institute for Mathematical Sciences, Kyoto University, Oiwake-Kitashirakawa, Sakyo, Kyoto, 606-8502, Japan \& Institut de Math\'{e}matiques de Jussieu-Paris Rive Gauche, Universit\'{e} de Paris, F-75013, Paris, France}
\email{rfujita@kurims.kyoto-u.ac.jp}

\author[K.~Murakami]{Kota Murakami}
\address[K.~Murakami]{Department of Mathematics, Kyoto University, Kitashirakawa Oiwake-cho, Sakyo-ku, Kyoto, 606-8502, Japan}
\email{k-murakami@math.kyoto-u.ac.jp}

\begin{document} 
\maketitle
\begin{abstract}
We give an interpretation of the $(q,t)$-deformed Cartan matrices of finite type and their inverses in terms of bigraded modules over the generalized preprojective algebras of Langlands dual type in the sense of Gei\ss-Leclerc-Schr\"{o}er~[Invent.~math.~{\bf{209}} (2017)].
As an application, we compute the first extension groups between the generic kernels introduced by Hernandez-Leclerc~[J.~Eur.~Math.~Soc.~{\bf 18} (2016)], and propose a conjecture that their dimensions coincide with the pole orders of the normalized $R$-matrices between the corresponding Kirillov-Reshetikhin modules. 
\end{abstract}
\tableofcontents

\section*{Introduction}
Let $\fg$ be a complex finite-dimensional simple Lie algebra and $C=(c_{ij})_{i,j \in I}$ its Cartan matrix.
In their seminal work~\cite{FR98}, Frenkel-Reshetikhin introduced a certain two-parameter deformation $C(q,t)$ of the Cartan matrix $C$, which we call the \emph{$(q,t)$-deformed Cartan matrix}.
It is used to define a two-parameter deformation $\mathcal{W}_{q,t}(\fg)$ of the $\mathcal{W}$-algebra associated with $\fg$ (in type $\mathrm{A}$, it was previously introduced by \cite{AKOS, FF}), which is expected to ``interpolate" the representation ring of the quantum affine algebra $U_{q}'(\widehat{\fg})$ and that of its Langlands dual $U_t'({}^L\widehat{\fg})$ through appropriate specializations of the parameters $q$ and $t$.
Indeed, the specialization $C(q)\coloneqq C(q, 1)$ at $t=1$, often called the quantum Cartan matrix, or rather its inverse $\tC(q)$ appear ubiquitously as key combinatorial ingredients in the representation theory of the quantum affine algebra $U'_q(\widehat{\fg})$ and the Yangian $Y(\fg)$.
For example, they play an important role in the study of $q$-characters~\cite{FR,FM} and quantum Grothendieck rings~\cite{Nakqt,VV,Her,HL15}, the description of the commutative part of the universal $R$-matrices~\cite{KT2,GTL} and denominator formulas of the normalized $R$-matrices~\cite{Fuj,FO}.  
We also refer to~\cite{GW,CL} for their more recent appearances.

Among the properties of the matrix $\tC(q)$, it is remarkable that the coefficients of its formal Taylor expansion at $q=0$ show certain periodicity and positivity. At first,
they had been understood by experts merely as a consequence of case-by-case computation.
Recently, a unified proof using Weyl group combinatorics was given in \cite{HL15,Fuj} for simply-laced type, and in \cite{FO} for general type. 
See also \cite{GW} for another proof. 

On the other hand, it is known that some kinds of $q$-analogues of (symmetric) generalized Cartan matrices (GCM) appear as homological invariants arising from graded modules over certain classes of associative algebras (\eg \cite{HK,ET,IQ,Kels}). 
The aim of the present paper is to give a new interpretation of the deformed Cartan matrix $C(q,t)$ and its inverse $\tC(q,t)$ following a similar philosophy. 
Namely, we consider a (bi)graded version of the \emph{generalized preprojective algebras} introduced by Gei\ss-Leclerc-Schr\"{o}er~\cite{GLS} and study its relation to the deformed Cartan matrices. In their studies~\cite{GLS,GLS2,GLS3,GLS4,GLS5,GLS6}, they generalize several connections between the representation theory of quivers and Kac-Moody algebras associated with symmetric GCMs to symmetrizable settings, and they have introduced a $1$-Iwanaga-Gorenstein algebra and its ``double" called the generalized preprojective algebra. These algebras are given by specific quivers with relations, which depend on the choice of symmetrizable GCM and its symmetrizer. 
If (and only if) the GCM is of finite type, the generalized preprojective algebra becomes finite-dimensional and self-injective over a base field.
Note that the Lie algebra $\fg$ in the present paper is Langlands dual to the one in the works of Gei\ss-Leclerc-Schr\"oer and the previous works~\cite{M2,M} of the second named author.
In particular, our Cartan matrix is transposed to the one in \cite{GLS}. Remarkably, in their work~\cite{GLS6}, the root system of Langlands dual type has also appeared: They have classified finite bricks over their $1$-Iwanaga-Gorenstein algebra in terms of Schur roots associated with the transposed GCM in a viewpoint of the $\tau$-tilting theory~\cite{Asa,DIJ}.

As mentioned in \cite[\S 1.7.1]{GLS}, the definition of generalized preprojective algebras is inspired in part by the previous work~\cite{HL} of Hernandez-Leclerc on the representation theory of quantum affine algebras.
Indeed, a graded version of the generalized preprojective algebra of finite type already appeared there in the disguise of the Jacobian algebra $J_{\Gamma, W}$ associated with a certain infinite quiver $\Gamma$ with potential $W$. (The same quiver with potential also appeared in the context of theoretical physics, cf.~\cite{CdZ,Cec}.)
This fact also motivates us to study a relationship between the deformed Cartan matrices and (bi)graded modules over the generalized preprojective algebras.

Let us explain our results in the graded (i.e., $t=1$) setting for simplicity.
Let $r \in \{1,2,3\}$ be the lacing number of $\fg$ and $D = \mathrm{diag}(d_i \mid i \in I)$ the minimal left symmetrizer of $C$.  
They define the generalized preprojective algebra $\Pi$ over an arbitrary field. We endow $\Pi$ with a grading following~\cite{HL}.
In the main body of this paper, we actually endow $\Pi$ with a bigrading and work with bigraded $\Pi$-modules to establish the $(q,t)$-versions of the results below (see \S \ref{GPAdef} for definitions). 

There is a maximal indecomposable iterated self-extension $E_i$ of the simple $\Pi$-module $S_i$ associated with each index $i \in I$, called the generalized simple module.  
Though the algebra $\Pi$ has infinite global dimension and actually $E_i$ has infinite projective dimension, the \emph{graded} Euler-Poincar\'e pairing $\langle E_i, S_j \rangle_{q}$ makes sense as a formal Laurent series in $q$.
This is an advantage of our graded setting.
Indeed, in terms of the matrix $C(q)$, we obtain
\begin{equation}\label{eq:ESq}
 \langle E_i, S_j \rangle_{q} = \frac{q^{d_i}}{1-q^{2rh^\vee}} \left( C_{ij}(q) - q^{rh^\vee} C_{i^* j}(q)\right), \end{equation}
in $\Z(\!(q)\!)$ for each $i, j \in I$, where $h^\vee$ is the dual Coxeter number of $\fg$, and $i \mapsto i^*$ is the involution of $I$ given by the longest element of the Weyl group.
The denominator $1-q^{2rh^\vee}$ reflects the fact that the projective resolution of $E_i$ is periodic up to grading shift by degree $2rh^\vee$.
We prove this fact by using a kind of reflection functors for graded $\Pi$-modules and its interpretation by the braid group action studied by Bouwknegt-Pilch~\cite{BP} and Chari~\cite{Cha}.

On the other hand, we consider a certain $\Pi$-submodule $\bar{I}_i$ of the $i$-th indecomposable injective module $I_i$, which is dual to $E_i$ with respect to the graded Euler-Poincar\'e pairing, i.e., $\langle E_i, \bar{I}_j \rangle_{q,t} = \delta_{ij}$.
In the terminology of \cite{HL}, this $\bar{I}_i$ can be identified with the generic kernel corresponding to the $i$-th fundamental $U'_q(\widehat{\fg})$-module.
By an easy homological investigation, we deduce the following Theorem~\ref{THM:main} from the above formula \eqref{eq:ESq}. Note that we obtain a simple explanation for the aforementioned periodicity and positivity of the matrix $\tC(q)$ as an immediate consequence of Theorem \ref{THM:main}. 

\begin{THM}[$\Leftarrow$ Theorem~\ref{Thm:goal} \& Corollary~\ref{Cor:goal}] \label{THM:main} For any indices $i,j \in I$, let 
$\tC_{ij}(q) = \sum_{u\ge 0} \tc_{ij}(u)q^u \in \Z[\![q]\!]$
denote the formal Taylor expansion at $q=0$ of the $(i,j)$-entry of the matrix $\tC(q)$. Then, we have 
\begin{gather*} \dim_q e_i \bar{I}_j = q^{-d_j}\sum_{u=0}^{rh^\vee} \tc_{ij}(u)q^u, \allowdisplaybreaks \\
\tC_{ij}(q) = \frac{q^{d_j}}{1-q^{2rh^\vee}} \left( \dim_q e_i \bar{I}_j - q^{rh^\vee} \dim_q e_i \bar{I}_{j^*} \right).
\end{gather*}
\end{THM}

We believe that Theorem~\ref{THM:main} shows one aspect of the relationship between the representation theory of generalized preprojective algebras and that of affine quantum groups, especially in non-simply-laced type.
It seems interesting to relate our result with other aspects studied in~\cite{YZ,NW} in a more geometric manner. 

As an application of Theorem~\ref{THM:main}, we compute the first extension groups between the generic kernels corresponding to the Kirillov-Reshetikhin (KR) modules in the sense of~\cite{HL}.
These generic kernels are certain modules over the Jacobian algebra $J_{\Gamma, W}$, which can be obtained as non-trivial self-extensions of the fundamental generic kernels $\bar{I}_i$.
By Hernandez-Leclerc's geometric character formula, the $F$-polynomials of these generic kernels give the $q$-characters of the KR modules.   
In this paper, we show that the dimension of the first extension groups between these generic kernels can be written in terms of the coefficients $\tc_{ij}(u)$. 
Then we compare the result with the conjectural denominator formula of the normalized $R$-matrices between the KR modules proposed in the previous work~\cite{FO} of Se-jin Oh and the first named author.
As a result, we find several pieces of evidence for the following conjecture, which we newly propose in this paper as a generalization of the conjectural denominator formula in \cite{FO}. 

\begin{CONJ}[$=$ Conjecture~\ref{Conj:main2}] \label{CONJ}
The dimensions of the first extension groups between these generic kernels coincide with the pole orders of the normalized $R$-matrices between the corresponding Kirillov-Reshetikhin modules.
\end{CONJ} 

Since the original definition of $q$-characters in \cite{FR} involves the $R$-matrices, it seems natural that the generic kernels should contain some information of the $R$-matrices in light of the geometric character formula. Conjecture~\ref{CONJ} suggests one of concrete connections between the generic kernels and the $R$-matrices.

We also believe that Conjecture~\ref{CONJ} can be understood from a cluster-theoretic point of view as follows:
It is known that certain algebraic structures arising from the quiver with potential $(\Gamma, W)$ give rise to an \emph{additive} categorification of the cluster algebra associated with $\Gamma$ (see \eg \cite{DWZ1,DWZ2,Ami,FK,KY,BIRS}).
On the other hand, essentially the same cluster algebra is \emph{monoidally} categorified by a category of modules over the quantum affine algebra $U'_q(\widehat{\fg})$, as was originally conjectured in \cite{HL} and proved very recently by Kashiwara-Kim-Oh-Park~\cite{KKOP,KKOP2}. 
Roughly speaking, the notion of a cluster corresponds to a maximal mutually $\Ext^1$-vanishing collection of indecomposable rigid objects in an additive categorification, while it corresponds to a maximal mutually commuting collection of prime real simple objects in a monoidal categorification.    
Note that the latter commutativity is essentially equivalent to the regularity of the corresponding normalized $R$-matrices. 
Furthermore, it is also notable that normalized $R$-matrices with non-trivial poles play an important role when we monoidally categorify the exchange relations.
On the other hand, in an additive categorification, the exchange relations correspond to some non-trivial extensions of indecomposable rigid objects.   
Thus, one may interpret Conjecture~\ref{CONJ} as a suggestion of a partial coincidence of the numerical characteristics between additive and monoidal categorifications of the same cluster algebra.   

Recently, certain deformed Cartan matrices of more general types are considered in a study of theoretical physics \cite{KP}. This kind of generalized deformed Cartan matrices have appeared in the representation theory of deformed $\mathcal{W}$-algebras beyond finite types (cf.~\cite{FJMV}). So, it seems interesting if we can find such a natural generalization of deformed Cartan matrix as a homological invariant of the generalized preprojective algebra other than of finite type. 
We plan to come back to this problem in the future. 

\subsection*{Organization}
This paper is organized as follows.
In Section~\ref{Sec:Cartan}, we define the $(q,t)$-deformed Cartan matrices of finite type and discuss the braid group action arising from them.
Section~\ref{Sec:GPA} is a preliminary on the generalized preprojective algebras of finite type in the bigraded setting.  
Section~\ref{Sec:main} is the main part of this paper. 
We discuss bigraded projective resolutions of the generalized simple modules, $E$-filtrations of projective modules and bigraded Euler-Poincar\'e pairings.
We establish the $(q,t)$-versions of the formula \eqref{eq:ESq} and Theorem~\ref{THM:main} in \S \ref{Ssec:main}. 
In Section~\ref{Sec:Rem}, we switch to consider the graded modules (rather than bigraded modules) and prepare some materials we need in the sequel. 
In Section~\ref{Sec:K}, we study the generic kernels corresponding to the KR modules.
We compute all the first extension groups between them explicitly in terms of the matrix $\tC(q)$.
In \S \ref{Ssec:conj}, we compare our computation with the conjectural denominator formula of normalized $R$-matrices, and propose Conjecture~\ref{CONJ}. 
Computations for a few exceptions are postponed in Appendix~\ref{Apx}. 

\subsection*{Conventions and notation}
Throughout this paper, we fix an arbitrary commutative field $\kk$. 
We always refer to an algebra as a (not necessarily unital) associative algebra over $\kk$.
For an algebra $A$, we denote by $A^\op$ (resp.~$A^\times$) its opposite algebra (resp.~multiplicative group of invertible elements). When we refer to a module over an algebra $A$, it means a left $A$-module unless specified otherwise. We naturally identify a right $A$-module with an $A^{\op}$-module. We say that a subcategory $\mathcal{C}$ of an exact category $\mathcal{E}$ is extension-closed (or closed under extensions) if for any conflation $0 \rightarrow L \rightarrow M \rightarrow N \rightarrow 0$, if $L\in \mathcal{E}$ and $N \in \mathcal{E}$, then so does $M$. 
By a grading, we always mean a $\Z$-grading, and hence by a bigrading, we mean a $\Z^2$-grading. 
For a statement $P$, we set $\delta(P)$ to be $1$ or $0$
according that $P$ is true or not.
As a special case, we set $\delta_{xy} \seq \delta(x=y)$ (Kronecker's delta).   

\section{Deformed Cartan matrices}
\label{Sec:Cartan}

In this section, we introduce the $(q,t)$-deformed Cartan matrix of finite type following Frenkel-Reshetikhin~\cite{FR}. 
We also discuss the braid group actions arising from them following Bouwknegt-Pilch~\cite{BP} and Chari~\cite{Cha}. 

\subsection{Notation}\label{subsec:notation}

Let $\fg$ be a complex finite-dimensional simple Lie algebra and let $C=(c_{ij})_{i,j \in I}$ be its Cartan matrix.
We set $n \seq \#I$. 
For any distinct $i,j \in I$, we write $i \sim j$ if $c_{ij} < 0$. 
Let $r$ denote the lacing number of $\fg$, which is defined by
\[
r \seq \begin{cases}
1 & \text{if $\fg$ is of type $\mathrm{A}_{n},\mathrm{D}_{n}$, or $\mathrm{E}_{6,7,8}$}, \\
2 & \text{if $\fg$ is of type $\mathrm{B}_{n}, \mathrm{C}_{n}$, or $\mathrm{F}_{4}$}, \\
3 & \text{if $\fg$ is of type $\mathrm{G}_{2}$}.
\end{cases}
\]
We say that $\fg$ is of simply-laced type if $r=1$.  
Let $D=\mathrm{diag}(d_i \mid i \in I)$ denote the minimal left symmetrizer of $C$ (see Table~\ref{table:num}).
Namely, $(d_i)_{i \in I}$ is the unique $I$-tuple of positive integers which are mutually coprime and satisfy $d_ic_{ij}=d_jc_{ji}$ for all $i,j \in I$.
We have $d_i \in \{1,r\}$ for any $i \in I$,
and the following relation holds:
\begin{equation} \label{by_ceil}
c_{ij} = \begin{cases}
2 & \text{if $i=j$}, \\
-\lceil d_j / d_i \rceil & \text{if $i \sim j$}, \\
0 & \text{else.}
\end{cases}
\end{equation}
We set $b_{ij} \seq d_ic_{ij}$ for each $i,j \in I$. 
Note that we have $b_{ij} = b_{ji}$ for any $i,j \in I$, and
\begin{equation}
b_{ij} = -d_i \lceil d_j / d_i\rceil = -\max(d_i, d_j) \qquad \text{if $i \sim j$.}
\end{equation}

\begin{table}[htbp]\centering
 { \arraycolsep=1.6pt\def\arraystretch{1.0}
\begin{tabular}{|c|c|c|c|c|}
\hline
$r$ & type of $\fg$ & $(d_i)_{i \in I}$ & $h$ & $h^{\vee}$  \\
\hline
\hline
& $\mathrm{A}_{n}$ & $(1, \ldots, 1)$& $n+1$ & $n+1$  \\
 $1$ & $\mathrm{D}_{n}$ & $(1, \ldots, 1)$ & $2n-2$ &$2n-2$  \\
  & $\mathrm{E}_{6,7,8}$ & $(1, \ldots, 1)$ & $12, 18, 30$ & $12, 18, 30$ \\
\hline
& $\mathrm{B}_{n}$ & $(2, \ldots, 2, 1)$ & $2n$ & $2n-1$ \\
$2$ & $\mathrm{C}_{n}$ & $(1, \ldots, 1, 2)$ & $2n$ &  $n+1$ \\
 & $\mathrm{F}_{4}$ & $(2,2,1,1)$ & $12$ & $9$ \\
\hline
$3$ & $\mathrm{G}_{2}$& $(3,1)$ & $6$ & $4$ \\
\hline
\end{tabular}
  }\\[1.5ex]
\caption{Basic numerical data} \label{table:num} 
\end{table}

\begin{Rem} \label{Ldual}
Note that the matrix $rD^{-1} = \mathrm{diag}(r/d_i \mid i \in I)$ gives the minimal left symmetrizer of the transposed Cartan matrix ${}^{\mathtt{t}}C = (c_{ji})_{i,j \in I}$.
\end{Rem}

Let $\alpha_i$ be the $i$-th simple root of $\fg$ for each $i \in I$ and 
$\sQ \seq \bigoplus_{i \in I}\Z \alpha_i$ the root lattice.
For each $i \in I$, the $i$-th simple reflection $s_i$ is defined to be the $\Z$-linear transformation of $\sQ$ given by $s_i (\alpha_j) = \alpha_j - c_{ij}\alpha_i$ for any $j \in I$. 
The Weyl group $W_\fg$ of $\fg$ is the subgroup of $\Aut_\Z(\sQ)$ generated by the simple reflections $\{ s_i \}_{i \in I}$.
The pair $(W_\fg, \{s_i\}_{i \in I})$ forms a finite Coxeter system. 
Let $w_0$ denote the longest element of $W_\fg$. 
It induces the involution $i \mapsto i^*$ of the set $I$ by $w_0(\alpha_i) = - \alpha_{i^*}$.
This involution gives the non-trivial automorphism of the Dynkin diagram of $\fg$ if and only if $\fg$ is either of type $\mathrm{A}_n$ (for any $n$), $\mathrm{D}_n$ (for $n$ odd) or $\mathrm{E}_6$.

\subsection{Deformed Cartan matrices}\label{subsection:DCM}

Let $q$ and $t$ be indeterminates. 
For an integer $k$, we set 
\[
[k]_{q} \seq  \frac{q^{k}-q^{-k}}{q-q^{-1}},
\]
which is an element of $\Z[q^{\pm 1}]$.
Following \cite{FR98}, we consider the $\Z[q^{\pm 1}, t^{\pm 1}]$-valued $I \times I$-matrix $C(q,t)$ whose $(i,j)$-entry $C_{ij}(q,t)$ is given by
\[
C_{ij}(q,t) \seq\begin{cases}
q^{d_i}t^{-1} + q^{-d_i}t & \text{if $i=j$}, \\
[c_{ij}]_q & \text{if $i\neq j$}. 
\end{cases}
\]
Specializing $(q,t)$ to $(1,1)$, we get $C(1,1) = C$.
Thus the matrix $C(q,t)$ gives a $(q,t)$-deformation of the Cartan matrix $C$ of $\fg$.
Specializing $t$ to $1$, we define $C(q)\seq C(q,1)$, which is sometimes referred to as the quantum Cartan matrix of $\fg$.
When $\fg$ is of simply-laced type, we have $C(q,t) = C(qt^{-1})$.
For general $\fg$, we have $[d_i]_q C_{ij}(q,t) = [d_i c_{ij}]_q$ whenever $i\neq j$, and hence the matrix $([d_i]_q C_{ij}(q,t))_{i,j \in I}$ is symmetric. 

Let $q^{\pm D} \seq \mathrm{diag}(q^{\pm d_i} \mid i \in I)$.
We see that the matrix $C(q,t)$ can be written in the form
\[ C(q,t) = (\id - A(q,t)) q^{-D}t\]
with some $A(q,t) \in qt^{-1}\cdot\mathfrak{gl}_I(\Z[q,t^{-1}])$ (cf.~\cite[Lemma 4.3]{FO}).
Therefore $C(q,t)$ is invertible as an element of $GL_I(\Z[\![q,t^{-1}]\!][(qt^{-1})^{-1}])$.
We write $\tC(q,t)$ for its inverse. 
With the above notation, we have
\begin{equation} \label{eq:tC} 
\tC(q,t) = q^D t^{-1} \left( \id + \sum_{k =1}^{\infty} A(q,t)^k \right). 
\end{equation}
For each $i,j\in I$, we express the $(i,j)$-entry $\tC_{ij}(q,t)$ of the matrix $\tC(q,t)$ as
\[ \tC_{ij}(q,t) = \sum_{u,v \in \Z} \tc_{ij}(u,v)q^u t^v\]
with $\tc_{ij}(u,v) \in \Z$.
The equation \eqref{eq:tC} implies the following.

\begin{Lem} \label{Lem:tc}
For each $i,j \in I$, we have 
\begin{enumerate}
\item $\tc_{ij}(u,v)=0$ if $u \le d_i$ or $v \ge -1$ but $(u,v) \neq (d_i, -1)$,
\item $\tc_{ij}(d_i,-1) = \delta_{ij}$.\qedhere
\end{enumerate}
\end{Lem}

\begin{Ex}
Let $\fg$ be of type $\mathrm{C}_2$ $(= \mathrm{B}_2)$ and we set $I = \{1,2\}$ with $(d_1,d_2) = (1,2)$.
By definition, we have
\[
C = \begin{pmatrix}2&-2\\-1&2\end{pmatrix} \quad \text{and} \quad
C(q,t) = \begin{pmatrix} qt^{-1} + q^{-1}t & -(q+q^{-1}) \\ -1 & q^2t^{-1} + q^{-2}t \end{pmatrix}.
\]
Since $\mathop{\mathrm{det}}C(q,t) = q^3t^{-2}+q^{-3}t^2 = q^{-3}t^2(1+q^6t^{-4})$, we have
\[
\tC(q,t) = \frac{q^3t^{-2}}{1+q^6t^{-4}}\begin{pmatrix}q^2t^{-1}+q^{-2}t&q+q^{-1}\\1&qt^{-1}+q^{-1}t\end{pmatrix}. 
\]
Here, we observe that the expansion coefficients exhibit a quasi-periodicity ($\tc_{ij}(u+6, -v-4) = -\tc_{ij}(u,-v)$ for any $i,j \in I$ and $u, v \in \Z_{\ge 0}$) and that the entries of the matrix $(q^3t^{-2}+q^{-3}t^2)\tC(q,t)$ are palindromic polynomials (i.e., invariant under the exchange $(q,t) \leftrightarrow (q^{-1}, t^{-1})$) with non-negative coefficients. 
Later, we show that these remarkable combinatorial properties hold for general $\fg$ (see Corollary~\ref{Cor:properties}).
\end{Ex}

\subsection{Braid group action}\label{subsec:Braid}
We consider an $n$-dimensional $\Q(q,t)$-vector space $\hq$ given by 
\[ \hq \seq \Q(q,t)\otimes_{\Z}\sQ = \bigoplus_{i \in I} \Q(q,t)\alpha_i.\]
We endow $\hq$ with a non-degenerate symmetric $\Q(q,t)$-bilinear pairing $(-,-)_{q,t}$ by
\[ (\alpha_i, \alpha_j)_{q,t} \seq [d_i]_q C_{ij}(q,t)\]
for each $i,j \in I$.
Let $\{\alpha_i^\vee\}_{i \in I}$ be another basis of $\hq$ defined by $\alpha_i^\vee \seq q^{-d_i}t\alpha_i/[d_i]_q$. 
We have $(\alpha_i^\vee, \alpha_j)_{q,t} = q^{-d_i}tC_{ij}(q,t)$ for any $i,j \in I$.
Let $\{ \varpi_i^\vee \}_{i\in I}$ denote the dual basis of $\{\alpha_i\}_{i \in I}$ with respect to $(-,-)_{q,t}$. 
We also consider the element $\varpi_i \seq [d_i]_q \varpi_i^\vee$ for each $i \in I$.
It is thought of a deformation of the $i$-th fundamental weight.
With our conventions, we have $(\varpi_i, \alpha_j^\vee)_{q,t} = \delta_{ij}q^{-d_i}t$ for any $i, j \in I$, and
\[ \alpha_i = \sum_{j \in I}C_{ji}(q,t)\varpi_j, \qquad
\alpha_i^\vee = q^{-d_i}t\sum_{j \in I} C_{ij}(q,t) \varpi_j^\vee
\]
for each $i \in I$.

Let $B_\fg$ denote the braid group associated with the Coxeter system $(W_\fg, \{s_i\}_{i \in I})$. 
It is the group presented by the generators $\{T_i\}_{i \in I}$ which subject to the relations:
\begin{alignat*}{2}
T_i T_j &= T_j T_i &\qquad & \text{if $c_{ij}=0$}, \\
T_i T_j T_i &= T_j T_i T_j && \text{if $c_{ij}c_{ji}=1$}, \\
(T_i T_j)^r &= (T_j T_i)^r && \text{if $c_{ij}c_{ji}=r>1$}. 
\end{alignat*}
For any $w \in W_\fg$, we choose a reduced expression $w = s_{i_1}s_{i_2} \cdots s_{i_l}$ and set $T_w \seq T_{i_1} T_{i_2} \cdots T_{i_l} \in B_{\fg}$.
The element $T_w$ is independent of the choice of reduced expression of $w$.

Following \cite[\S 3]{BP} and \cite[\S 3]{Cha}, we define an action of the braid group $B_\fg$ on the $\Q(q,t)$-vector space $\hq$ by
\begin{equation} \label{Baction}
T_i \lambda \seq \lambda - (\alpha_i^\vee, \lambda)_{q,t} \alpha_i 
\end{equation}
for any $\lambda \in \hq$.
Equivalently, in terms of the basis $\{ \alpha_i \}_{i \in I}$, we have
\begin{equation} \label{Troot}
T_i^{\pm 1} \alpha_j = \alpha_j - q^{\mp d_i}t^{\pm 1}C_{ij}(q,t) \alpha_i.
\end{equation}
Thus the action \eqref{Baction} is a $(q,t)$-deformation of the action of the Weyl group $W_\fg$ on $\fh^*$ given in \S \ref{subsec:notation}.
On the other hand, for the basis $\{\alpha^\vee_i \}_{i \in I}$, we have
\begin{equation} \label{Tcoroot}
T_i \alpha_j^\vee = \alpha_j^\vee - q^{-d_j}tC_{ji}(q,t)\alpha_i^\vee.
\end{equation}

\begin{Lem}\label{lem:eqpair}
For any $\lambda, \mu \in \hq$ and $i \in I$, we have $(T_i\lambda, \mu)_{q,t} = (\lambda, T_i\mu)_{q,t}$.
\end{Lem}
\begin{proof} A straightforward computation
\[(T_i\lambda, \mu)_{q,t} = (\lambda, \mu)_{q,t} - \frac{q^{-d_i}t}{[d_i]_q}(\alpha_i, \lambda)_{q,t} (\alpha_i, \mu)_{q,t} = (\lambda, T_i \mu)_{q,t} \]
yields the assertion.
\end{proof}

\begin{Lem} \label{Lem:quadrant}
For any $i \in I$, the above action of $T_i$ preserves the subset $\bigoplus_{j \in I} \Z[q^{-1}, t] \varpi_j$.
\end{Lem}
\begin{proof}
Let $\lambda = \sum_{j \in I} \lambda_j \varpi_j$ be an arbitrary element of $\hq$ and $i \in I$. 
If we write $T_i \lambda = \sum_{j \in I} \mu_j \varpi_j$, the equation \eqref{Baction} is expressed as
\[
\mu_j = \begin{cases}
\lambda_j & \text{if $c_{ij}=0$}, \\
\lambda_j + q^{-d_i}t\lambda_i & \text{if $c_{ji}=-1$}, \\
\lambda_j + (q^{-1} + q^{-3})t \lambda_i & \text{if $c_{ji}=-2$}, \\
\lambda_j + (q^{-1} + q^{-3} + q^{-5})t\lambda_i & \text{if $c_{ji}=-3$}, \\
- q^{-2d_i}t^2 \lambda_i & \text{if $j=i$}.
\end{cases}
\] 
This proves the assertion.
\end{proof}
Let $h$ and $h^\vee$ be the Coxeter number and the dual Coxeter number of $\fg$ respectively (see Table~\ref{table:num}).
We write $\nu$ for the $\Q(q,t)$-linear involution on $\hq$ given by $\nu(\alpha_i) = \alpha_{i^*}$. 
\begin{Thm}[Bouwknegt-Pilch, Chari]\label{Thm:Tw_0}
For any $\lambda \in \hq$, we have 
\begin{equation} \label{eq:w0}
T_{w_0} \lambda = - q^{-rh^\vee}t^h \nu(\lambda). 
\end{equation}
\end{Thm}    
\begin{proof}
The assertion is stated in \cite[(3.45)]{BP} without a proof.
Here we give a detailed proof relied on its $q$-version~\cite{Cha} for completeness.
It is enough to show that the relation \eqref{eq:w0} holds when $\lambda$ belongs to the integral lattice $\sQ_{q,t} \seq \Z[q^{\pm 1}, t^{\pm 1}]\otimes \sQ$.
Note that the $B_\fg$-action preserves $\sQ_{q,t}$ and, for any $\lambda \in \sQ_{q,t}$, one can consider its specialization $[\lambda]_{t=1} \in \sQ_q \seq \Z[q^{\pm 1}]\otimes \sQ$.
With a $\Z[q^{\pm 1}, t^{\pm 1}]$-endomorphism $f$ of $\sQ_{q,t}$, we associate the $\Z[q^{\pm 1}]$-linear endomorphism $[f]_{t=1}$ of $\sQ_q$ defined by $[f]_{t=1}\lambda \seq [f(\lambda)]_{t=1}$ for $\lambda \in \sQ_q$.
We linearly extend $[f]_{t=1}$ to be a $\Z[q^{\pm 1}, t^{\pm 1}]$-endomorphism of $\sQ_{q,t}$.
It follows that $[f\circ g]_{t=1} = [f]_{t=1} \circ [g]_{t=1}$. 

It is known that the relation \eqref{eq:w0} holds if specialized at $t=1$ by \cite[Proposition 4.1]{Cha} combined with \cite[Lemma 6.8]{FM}. (Note that $q$ in \cite[\S3]{Cha} is our $q^{-1}$. Remarkably, this proof uses representation theory of the quantum affine algebras. An alternative, case-by-case combinatorial proof is suggested in \cite[Proposition 3.6]{CM}, while $q^{h^\vee}$ therein should be replaced with $q^{rh^\vee}$.)
More precisely, we have
\begin{equation} \label{eq:w0q}
[T_{w_0}]_{t=1}\lambda = -q^{-rh^\vee}\nu(\lambda)
\end{equation}
for any $\lambda \in \sQ_{q,t}$.
When $\fg$ is of simply-laced type, the relation \eqref{eq:w0} follows from \eqref{eq:w0q} because $h=h^\vee$ and we can replace $(q,t)$ with $(qt^{-1},1)$.  
Now we assume that $\fg$ is of non-simply-laced type.
In this case, the Coxeter number $h$ is even and $\nu = \id$. 
We choose a decomposition $I = J \sqcup J'$ such that $i \sim j$ implies $(i,j) \in J \times J'$ or $(i,j) \in J' \times J$.
It yields a Coxeter element $c=w w' \in W_\fg$, where $w \seq \prod_{j \in J} s_j, w' \seq \prod_{j \in J'}s_j$.
It is well-known that we have $w_0 = c^{h/2}$ and $T_{w_0} = T_c^{h/2}= (T_wT_{w'})^{h/2}$ (see \cite[\S 3.17]{Hum} for example).
Let $t_J$ denote the $\Z[q^{\pm 1}, t^{\pm 1}]$-automorphism of $\sQ_{q,t}$ given by $t_J \alpha_i = t^{\delta(i \in J)} \alpha_i$.
Then, from the equation \eqref{Troot}, it follows that
\[ t_J^{-1} T_w t_J^{-1} \lambda = [T_w]_{t=1} \lambda, \qquad t_J T_{w'} t_J \lambda = t^2[T_{w'}]_{t=1} \lambda\]
for any $\lambda \in \sQ_{q,t}$.
Therefore, we have
\begin{align*} 
T_{w_0} \lambda &= t_J \left((t_J^{-1} T_w t_J^{-1}) (t_J T_{w'} t_J) \right)^{h/2} t_J^{-1} \lambda \allowdisplaybreaks \\
&= t^h t_J \left([T_w]_{t=1} [T_{w'}]_{t=1}\right)^{h/2} t_J^{-1} \lambda \allowdisplaybreaks \\
&= t^h t_J [T_{w_0}]_{t=1} t_J^{-1} \lambda.
\end{align*}
Applying \eqref{eq:w0q} to the right hand side, we obtain the desired relation \eqref{eq:w0}.  
\end{proof}

\begin{Rem}
As mentioned in \cite[(3.40)]{BP}, we can easily check that the relation 
\[(T_i - 1)(T_i + q^{2d_i}t^{-2})\lambda =0\] holds for any $i \in I$ and $\lambda \in \hq$.
Therefore, the above $B_\fg$-action on $\hq$ descends to an action of the Iwahori-Hecke algebra associated with $\fg$. 
When $\fg$ is of simply-laced type, this was also observed by Ikeda-Qiu~\cite[Proposition A.5]{IQ}.
Note that the $q$-deformed Cartan matrix $A_Q(q)$ in~\cite{IQ} slightly differs from our matrix $C(q)$. 
Indeed, $A_Q(q)$ depends on a Dynkin quiver $Q$.
However, this difference is not essential because we have $A_Q(q^2) = q^{1-\xi} C(q) q^{\xi}$ with $q^{\xi} = \mathrm{diag}(q^{\xi_i}\mid i \in I)$, where $\xi \colon I \to \Z$ is a height function of $Q$ (cf.~ Remark~\ref{Rem:potential} below).  
\end{Rem}

\section{Generalized preprojective algebras}
\label{Sec:GPA}

In this preliminary section, we fix our conventions and give a brief review on the generalized preprojective algebras of finite type by Gei\ss-Leclerc-Schr\"oer~\cite{GLS} in a bigraded setting. 

\subsection{Bigraded vector spaces} \label{Ssec:bgv}
By an abuse of notation, $q$ and $t$ denote the grading shift functors for bigraded $\kk$-vector spaces.
Namely, for a bigraded $\kk$-vector space $V = \bigoplus_{u,v \in \Z} V_{u,v}$, we define its shift $qV$ (resp.~$tV$) by $(qV)_{u,v} \seq V_{u-1,v}$ (resp.~$(tV)_{u,v} \seq V_{u,v-1}$) for any $u,v \in \Z$. 
For a Laurent polynomial $a(q,t) = \sum_{k, l \in \Z} a_{k,l} q^k t^l$ in $q,t$ with non-negative integer coefficients, we set
\[ V^{\oplus a(q,t)} \seq \bigoplus_{k, l \in \Z} (q^k t^l V)^{\oplus a_{k,l}}.\] 
When every bigraded piece $V_{u,v}$ is finite-dimensional, we define the bigraded dimension $\dim_{q,t} V$ to be a formal power series in $q^{\pm 1}, t^{\pm 1}$ given by
\[ \dim_{q,t} V \seq \sum_{u,v \in \Z} (\dim_\kk V_{u,v}) q^u t^v.\]
With the above notation, we have $\dim_{q,t} (X^{\oplus a(q,t)}) = a(q,t) \dim_{q,t} V$.
We also define the restricted dual $\bD(V)$ of a bigraded $\kk$-vector space $V$ by $\bD(V)_{u,v} \seq \Hom_\kk(V_{-u,-v}, \kk)$ for each $(u,v) \in \Z^2$.
If each $V_{u,v}$ is finite-dimensional, we naturally have $\bD^{2}(V) \cong V$ and 
\[\dim_{q,t} \bD(V) = \dim_{q^{-1}, t^{-1}} V.\] 

\subsection{Bigraded quivers and algebras}
A quiver is a directed graph. 
We understand it as a quadruple $Q=(Q_0, Q_1, \se, \tl)$, where $Q_0$ is the set of vertices, $Q_1$ is the set of arrows and $\se$ (resp.~$\tl$) is the map $Q_1 \to Q_0$ which assigns each arrow with its source (resp.~target).
By a bigraded quiver, we mean a quiver $Q$ endowed with a map $\deg \colon Q_1 \to \Z^2$. 

For a quiver $Q$, we set $\kk Q_0 \seq \bigoplus_{i \in Q_0} \kk e_i$ and $\kk Q_1 \seq \bigoplus_{\alpha \in Q_1} \kk \alpha$.  
We endow $\kk Q_0$ with a (possibly non-unital) $\kk$-algebra structure by $e_i \cdot e_j = \delta_{ij} e_i$ for any $i,j \in Q_0$, and $\kk Q_1$ with a $(\kk Q_0, \kk Q_0)$-bimodule structure by $e_i \cdot \alpha = \delta_{i, \tl(\alpha)} \alpha$ and $\alpha \cdot e_i = \delta_{i, \se(\alpha)} \alpha$ for any $i \in Q_0$ and $\alpha \in Q_1$. 
Then the path algebra of $Q$ is defined to be the tensor algebra $\kk Q \seq T_{\kk Q_0}(\kk Q_1)$.  
When $Q$ is bigraded, its path algebra $\kk Q$ naturally becomes a bigraded $\kk$-algebra.
Note that we necessarily have $\deg(e_i)=(0,0)$ for any $i \in Q_0$.

Let $A$ be a (possibly non-unital) $\kk$-algebra obtained as a quotient of the path algebra $\kk Q$ of a quiver $Q$.
We always mean by an $A$-module $M$ a left $A$-module satisfying $M = \bigoplus_{i \in Q_0} e_i M$.
Assume that $Q$ is bigraded and $A$ inherits the bigrading.
For bigraded $A$-modules $M$ and $N$, we denote by $\Hom_A(M,N)$ the space of homogeneous $A$-homomorphisms and by $\Ext^m_A(M,N)$ the homogeneous $m$-th extension group.
Then we define the bigraded $\kk$-vector spaces $\ghom_A(M,N)$ and $\gext^m_A(M,N)$ respectively by
\[\ghom_A(M,N)_{u,v} \seq \Hom_A(q^u t^v M, N) \quad \text{and} \quad \gext^m_A(M,N)_{u,v} \seq \Ext^m_A(q^u t^v M, N)\]   
for each $u,v \in \Z$.
We understand $\Ext_A^0(M,N) = \Hom_A(M,N)$ and $\gext_A^0(M,N) = \hom_A(M,N)$ as usual.

\subsection{Generalized preprojective algebras}
\label{GPAdef}

We keep the notation in Section~\ref{Sec:Cartan}. 
We consider a finite quiver $\tQ = (\tQ_0, \tQ_1, \se, \tl)$ for any $\fg$ given as follows: 
\begin{gather*}
\tQ_0 = I, \quad
\tQ_1 = \{ \alpha_{ij} \mid (i,j) \in I \times I, i \sim j\} \cup \{ \ep_i \mid i \in I \}, \\
\se(\alpha_{ij}) = j, \quad \tl(\alpha_{ij})= i, \quad \se(\ep_i)=\tl(\ep_i)=i.
\end{gather*}
We endow the quiver $\tQ$ with a bigrading by
\begin{equation}\label{GPAdeg}
\deg(\alpha_{ij}) \seq (b_{ij},1) = (- \max(d_i, d_j),1), \qquad \deg(\ep_i) \seq (b_{ii},0) = (2d_i,0).
\end{equation}
We also choose a sign $\omega_{ij} \in \{ 1, -1 \}$ for each $(i,j) \in I \times I$ with $i \sim j$ such that $\omega_{ij} = -\omega_{ji}$.   
Then, we define the $\kk$-algebra $\tPi$ to be the quotient of $\kk \tQ$ by the following two kinds of relations:
\begin{itemize} 
\item[(R1)] $\ep_i^{-c_{ij}} \alpha_{ij} = \alpha_{ij} \ep_j^{-c_{ji}}$ for any $i,j \in I$ with $i \sim j$;
\item[(R2)] $\displaystyle \sum_{j \in I: j\sim i}\sum_{k = 0}^{-c_{ij}-1}\omega_{ij} \ep_i^k \alpha_{ij} \alpha_{ji} \ep_i^{-c_{ij}-1-k} =0$ for each $i \in I$.
\end{itemize} 
Since the relations are homogeneous, the algebra $\tPi$ inherits the bigrading from $\kk\tQ$. 
Up to bigraded $\kk$-algebra isomorphism, the algebra $\tPi$ does not depend on the choice of the signs $\{\omega_{ij}\}_{i \sim j}$.
Therefore, we have suppressed its dependence from the notation.

Thanks to the relation (R1), the element
\[ \ep \seq \sum_{i \in I} \ep_i^{r/d_i}e_i\]
is central in $\tPi$.
Note that $\ep$ is homogeneous of degree $(2r,0)$.
For each positive integer $\ell$, we define the quotient algebra
\[\Pi(\ell) \seq \tPi / \ep^\ell \tPi.\]
In other words, the $\kk$-algebra $\Pi(\ell)$ is defined as the quotient of $\kk\tQ$ by the three kinds of relations: (R1), (R2) and 
\begin{itemize}
\item[(R3)] $\ep_i^{\ell r/d_i}=0$ for each $i \in I$.
\end{itemize}
The algebra $\Pi(\ell)$ inherits the bigrading from $\kk \tQ$.

\begin{Rem} \label{PiGLS}
The algebra $\Pi(\ell)$ is identical to the generalized preprojective algebra denoted by $\Pi({}^{\mathtt{t}}C, \ell r D^{-1}, \Omega)$ in \cite{GLS} and in the previous works \cite{M2,M} of the second named author (recall Remark~\ref{Ldual}), where $\Omega \seq \{ (i,j) \in I \times I \mid i \sim j, \omega_{ij} =1 \}$ is the orientation corresponding to $\{ \omega_{ij}\}_{i \sim j}$.
\end{Rem}

For each $i\in I$, let $\kk[\ep_i]$ be the ring of polynomials in $\ep_i$ bigraded by $\deg(\ep_i)=(2d_i,0)$. 
Given a positive integer $\ell$, we set $H_i(\ell) \seq \kk[\ep_i]/(\ep_i^{\ell r/d_i})$. 
For any $\Pi(\ell)$-module $M$, the subspace $e_i M$ becomes a module over $H_i(\ell)$ in the obvious way for each $i \in I$.  

\begin{Thm}[Gei\ss-Leclerc-Schr\"{o}er] \label{Thm:fd}
Let $\ell \in \Z_{>0}$. 
\begin{enumerate}
\item \label{Thm:fd:lf}
For any $i, j \in I$, the space $e_i \Pi(\ell) e_j$ is bigraded free of finite rank over the algebra $H_i(\ell)$.
In particular, the algebra $\Pi(\ell)$ is finite-dimensional over $\kk$.
\item \label{Thm:fd:nilp}
If $v > n(h+1)$, we have $\Pi(\ell)_{u,v} = 0$ for any $u \in \Z$. \qedhere
\end{enumerate}
\end{Thm}
\begin{proof}
These assertions follow from the results in \cite[\S 11]{GLS}.
\end{proof}

There is an anti-involution of bigraded $\kk$-algebras $\phi \colon \tPi \to \tPi^{\op}$ given by 
\[\phi(e_i) \seq e_i, \qquad \phi(\alpha_{ij}) \seq \alpha_{ji}, \qquad \phi(\ep_i) \seq \ep_i.\]
For a bigraded $\tPi$-module $M$, we always regard its restricted dual $\bD(M)$ as a bigraded left $\tPi$-module by twisting its natural right module structure with $\phi$. 
If each bigraded piece $M_{u,v}$ is finite dimensional, we have the natural isomorphism $\bD^{2}(M) \cong M$ of bigraded $\tPi$-modules.
The isomorphism $\phi$ naturally induces the isomorphism for the quotient $\Pi(\ell)$ for each $\ell \in \Z_{>0}$. 

\subsection{Bigraded modules}
In what follows, for a bigraded algebra $A$, we denote by $\Cc(A)$ the category of \emph{finitely generated} bigraded $A$-modules.
For $\ell \in \Z_{>0}$, we identify the category $\Cc(\Pi(\ell))$  with a full subcategory of $\Cc(\tPi)$ via the canonical quotient map $\tPi \to \Pi(\ell)$.
The category $\Cc(\Pi(\ell))$ is the same as the category of finite-dimensional bigraded $\Pi(\ell)$-modules by Theorem~\ref{Thm:fd}~\eqref{Thm:fd:lf}.
In particular, the duality functor $\bD$ induces a contravariant involutive auto-equivalence of $\Cc(\Pi(\ell))$.  

For each $i\in I$, let $S_i$ denote the associated simple module in $\Cc(\tPi)$ concentrated in bidegree $(0,0)$. 
Given $\ell \in \Z_{>0}$, we set $P_i(\ell) \seq \Pi(\ell)e_i = (\tPi/\tPi\ep_i^{\ell r/d_i})e_i$, 
which gives a projective cover of $S_i$ in the category $\Cc(\Pi(\ell))$.
Its restricted dual $I_i(\ell) \seq \bD(P_i(\ell))$ gives an injective hull of $S_i$ in $\Cc(\Pi(\ell))$.
We use the following fact later.
\begin{Thm}[{\cite[Corollary 12.7]{GLS},  \cite[paragraph after Theorem 3.18]{M}}] \label{Thm:P=I}

For each $\ell \in \Z_{>0}$, the algebra $\Pi(\ell)$ is a self-injective algebra (i.e., $\Pi(\ell)$ is an injective module as a $\Pi(\ell)$-module)
whose Nakayama permutation is identical to the involution $i \mapsto i^*$.
In particular, the injective module $I_i(\ell)$ is isomorphic to the projective module $P_{i^*}(\ell)$ up to bigrading shift for each $i \in I$. 
\end{Thm}

For $\ell \in \Z_{>0}$, we say that a bigraded $\Pi(\ell)$-module $M \in \Cc(\Pi(\ell))$ is \emph{locally free} if, for each $i \in I$, there is a Laurent polynomial $f_i = f_i(q,t) \in \Z_{\ge 0}[q^{\pm 1}, t^{\pm 1}]$ such that we have $e_i M \cong H_i(\ell)^{\oplus f_i}$ as bigraded $H_i(\ell)$-modules. 
We denote by $\Cc_{\lf}(\Pi(\ell)) \subset \Cc(\Pi(\ell))$ the full subcategory of locally free modules. Note that the category $\Cc_{\lf}(\Pi(\ell))$ is closed under taking extensions, kernel of epimorphisms and cokernel of monomorphisms (see \cite[Proof of Lemma 3.8]{GLS}).
By Theorem~\ref{Thm:fd}~\eqref{Thm:fd:lf}, we have $P_i(\ell), I_i(\ell) \in \Cc_{\lf}(\Pi(\ell))$ for any $\ell \in \Z_{>0}$ and $i \in I$.

For $\ell \in \Z_{>0}$ and $i \in I$, let $E_i(\ell)$ denote the maximal quotient of $P_i(\ell)$ such that $e_j E_i(\ell) = 0$ for any $j \neq i$. 
From the definition of $\Pi(\ell)$, we have $e_iE_i(\ell) \cong H_i(\ell)$ as bigraded $H_i(\ell)$-modules. 
We call $E_i(\ell)$ the \emph{generalized simple module} associated with $i$. 
Similarly, we can define the generalized simple modules in $\Cc(\Pi(\ell)^{\op})$, for which we use the symbol $E'_i(\ell)$.
We say that a module $M \in \Cc(\Pi(\ell))$ is \emph{$E$-filtered} if $M$ has a bigraded $\Pi(\ell)$-module filtration whose associated graded is a direct sum of some bigrading shifts of the generalized simple modules. We call this kind of filtrations \emph{$E$-filtrations}. 
Let $\Cc_E(\Pi(\ell))$ denote the full subcategory of $\Cc(\Pi(\ell))$ consisting of all the $E$-filtered modules. 
We have natural inclusions $\Cc_E(\Pi(\ell)) \subset \Cc_{\lf}(\Pi(\ell)) \subset \Cc(\Pi(\ell))$ of extension-closed subcategories for each $\ell \in \Z_{>0}$. 
\begin{Rem}
The category $\Cc_E(\Pi(\ell))$ is not an abelian category in general. In particular, $\Cc_E(\Pi(\ell))$ is an abelian category if and only if $C$ is symmetric and $\ell = 1$ (cf.~\cite[Corollary 2.8]{Eno} for a more general result). In this case, three categories $\Cc_E(\Pi(\ell)), \Cc_{\lf}(\Pi(\ell))$ and $\Cc(\Pi(\ell))$ coincide.
\end{Rem}
\subsection{Grothendieck groups} \label{Ssec:Grot}
Fix $\ell \in \Z_{> 0}$. 
We denote by $K(\Pi(\ell))$ the Grothendieck group of the abelian category $\Cc(\Pi(\ell))$.
For an object $M \in \Cc(\Pi(\ell))$, we write $[M] \in K(\Pi(\ell))$ for its isomorphism class. 
We endow $K(\Pi(\ell))$ with a structure of $\Z[q^{\pm 1}, t^{\pm 1}]$-module by setting $q [M] \seq [qM]$ and $t[M] \seq [tM]$ for any $M \in \Cc(\Pi(\ell))$.
Then, $K(\Pi(\ell))$ becomes a free $\Z[q^{\pm 1}, t^{\pm 1}]$-module with the basis $\{[S_i]\}_{i \in I}$.
We also consider its localization:
\[ K(\Pi(\ell))_{loc} \seq K(\Pi(\ell)) \otimes_{\Z[q^{\pm 1}, t^{\pm 1}]} \Q(q,t). \]
For simplicity, we also write $[M]$ for $[M] \otimes 1 \in K(\Pi(\ell))_{loc}$.
Since we have 
\[ [E_i(\ell)] = \frac{1-q^{2\ell r}}{1-q^{2d_i}} [S_i] \]
in $K(\Pi(\ell))$ for each $i \in I$, the set $\{ [E_i(\ell)]\}_{i \in I}$ forms a basis of $K(\Pi(\ell))_{loc}$.
Note that, if $M \in \Cc_{\lf}(\Pi(\ell))$ satisfies $[M] = \sum_{i\in I} f_i [E_i(\ell)]$ in $K(\Pi(\ell))_{loc}$, we have $f_i \in \Z_{\ge 0}[q^{\pm 1}, t^{\pm 1}]$ and $e_i M \cong H_i(\ell)^{\oplus f_i}$ for each $i \in I$. 

\subsection{The module $\bar{I}_i$}
For each $i \in I$, we define the bigraded $\tPi$-module $\bar{I}_i$ by
\[ \bar{I}_i \seq \bD((\tPi/\tPi \ep_i)e_i).\]
From the definition, it fits into the exact sequence 
\begin{equation} \label{eq:III}
0 \to \bar{I}_i \to I_i(\ell) \xrightarrow{\cdot \ep_i} q^{-2d_i}I_i(\ell)
\end{equation}
for any $\ell \in \Z_{>0}$.
In particular, $\bar{I}_i$ belongs to $\Cc(\Pi(\ell))$ for any $\ell \in \Z_{>0}$.

\begin{Lem} \label{Lem:extEI}
Let $\ell \in \Z_{>0}, m \in \Z_{\ge 0}$ and $i \in I$.
For any $M \in \Cc(\Pi(\ell))$, we have a natural isomorphism of bigraded vector spaces
\begin{equation} \label{eq:extMI}
\gext^{m}_{\Pi(\ell)}(M,\bar{I}_i) \cong \gext^{m}_{H_i(\ell)}(e_i M, \kk).
\end{equation}
In particular, we have
\begin{equation} \label{eq:extEI}
\gext^{m}_{\Pi(\ell)}(E_i(\ell),\bar{I}_j)
\cong \begin{cases}
\kk &\text{if $m=0$ and $i=j$}, \\
0 &\text{otherwise}.
\end{cases}
\end{equation}
\end{Lem}
\begin{proof}
Since $\ghom_{\Pi(\ell)}(M,I_i(\ell)) \cong \bD(e_i M)$, the exact sequence \eqref{eq:III} yields the isomorphism
\[ \ghom_{\Pi(\ell)}(M, \bar{I}_i) \cong \bD(e_i(M/\ep_i M)) \cong \ghom_{H_i(\ell)}(e_i M, \kk),\]
which is functorial in $M \in \Cc(\Pi(\ell))$.
This isomorphism extends to the desired isomorphism \eqref{eq:extMI} of the universal $\delta$-functors.
\end{proof}

\begin{Cor} \label{Cor:ME}
If $M \in \Cc_{\lf}(\Pi(\ell))$, we have
\[ [M] = \sum_{i \in I} \dim_{q^{-1}, t^{-1}} \ghom_{\Pi(\ell)}(M,\bar{I}_i) [E_i(\ell)] \]
in $K(\Pi(\ell))_{loc}$.
\end{Cor}

\section{Interpretation of deformed Cartan matrices}
\label{Sec:main}

In this section, we give a representation-theoretic interpretation of the $(q,t)$-deformed Cartan matrix $C(q,t)$ and its inverse $\tC(q,t)$ in terms of bigraded modules over the generalized preprojective algebras.
Along the way, we discuss bigraded projective resolutions of the generalized simple modules (\S \ref{Ssec:res}), $E$-filtrations of projective modules (\S \ref{Ssec:filt}) and bigraded Euler-Poincar\'e pairings (\S \ref{Ssec:EP}). 

Throughout this section, we fix a positive integer $\ell \in \Z_{>0}$ and consider $\Pi(\ell)$-modules only.
It turns out that all the results are essentially independent of this fixed $\ell$ except for Corollary~\ref{P=I}~\eqref{eq:P=I}.
For the sake of simplicity, we abbreviate $\Pi \seq \Pi(\ell)$ and $X_i \seq X_i(\ell)$ for each $i \in I$, where $X \in \{H, P, I, E, E' \}$.
Also, we set $\otimes_i \seq \otimes_{H_i}$ for each $i \in I$.

\subsection{Bigraded projective resolutions}
\label{Ssec:res}
In this subsection, we develop the projective resolution of $E_i$ in our bigraded category $\Cc(\Pi)$.
Following \cite[\S 5.1]{GLS}, for each $i, j \in I$ with $i \sim j$, we define the bigraded $(H_i, H_j)$-bimodule ${}_i H_j$ by
\[ {}_i H_j  \seq H_i \alpha_{ij} H_j  \subset \Pi. \]
It is free as a left $H_i$-module and free as a right $H_j$-module. 
Moreover, the relation $(R1)$ gives the following:
\[ {}_i H_j = \bigoplus_{k=0}^{-c_{ji}-1}H_i(\alpha_{ij}\varepsilon_j^k) 
    = \bigoplus_{k=0}^{-c_{ij}-1}(\varepsilon_i^k \alpha_{ij}) H_j. \]
In particular, we get the following lemma, which is essential to understand the relationship between the deformed Cartan matrices and the generalized preprojective algebras:
\begin{Lem}\label{lem:iHjCartan}
For $i, j\in I$ with $i\sim j$, we have two isomorphisms
\[
{}_{H_i}({}_i H_j) \cong H_i^{\oplus (-q^{-d_j}tC_{ji}(q,t))}, \qquad
({}_i H_j) {}_{H_j} \cong H_j^{\oplus (-q^{-d_i}tC_{ij}(q,t))}
\]
as bigraded left $H_i$-modules and as bigraded right $H_j$-modules respectively.
\end{Lem}

In the representation theory of (generalized) preprojective algebras, some bimodule resolutions developed in Brenner-Butler-King \cite{BBK} and Gei\ss-Leclerc-Schr\"oer \cite{GLS,GLS8} are very useful. 
Here we shall give a bigraded analogue of them by inspection.
Consider the following sequence of bigraded $(\Pi, \Pi)$-bimodules: 
\begin{equation} \label{eq:bmodres}
 \bigoplus_{i\in I} q^{-2d_i} t^2 \Pi e_i \otimes_i e_i \Pi
 \xrightarrow{\psi}
 \bigoplus_{i, j \in I: i \sim j} \Pi e_j \otimes_j {}_jH_i \otimes_i e_i \Pi 
 \xrightarrow{\varphi}
 \bigoplus_{i\in I} \Pi e_i \otimes_i e_i \Pi \rightarrow \Pi \rightarrow 0,
\end{equation}
where the morphisms $\psi$ and $\varphi$ are given by 
\begin{align*}
\psi(e_i \otimes e_i) &\seq \sum_{j \sim i}\sum_{k=0}^{-c_{ij}-1} \omega_{ij} \left(\ep_i^{k} \alpha_{ij} \otimes \alpha_{ji} \ep_i^{-c_{ij}-1-k}\otimes e_i + e_i \otimes \ep_i^{k} \alpha_{ij} \otimes \alpha_{ji} \ep_i^{-c_{ij}-1-k}\right), \\
\varphi(e_j \otimes x\otimes e_i) &\seq x \otimes e_i + e_j \otimes x.
\end{align*}
The other arrows $\bigoplus_{i\in I} \Pi e_i \otimes_i e_i \Pi \to \Pi \to 0$ are canonical. 
The relation (R2) ensures that the sequence~\eqref{eq:bmodres} forms a complex. 
For each $i \in I$, applying $(-)\otimes_{\Pi} E_i$ to \eqref{eq:bmodres} yields the following complex of bigraded (left) $\Pi$-modules:
\begin{equation} \label{res3}
P_2^{(i)} \xrightarrow{\psi^{(i)}}
P_1^{(i)}
\xrightarrow{\varphi^{(i)}}
P_0^{(i)} \to E_i \to 0,
\end{equation}
where 
\[ P_0^{(i)} = P_i, \qquad
P_1^{(i)} = \bigoplus_{j\sim i} P_j^{\oplus (-q^{-d_i}tC_{ij}(q,t))}, \qquad
P_2^{(i)} = q^{-2d_i}t^2 P_i
\]
by Lemma~\ref{lem:iHjCartan}.
The morphisms $\psi^{(i)}$ and $\varphi^{(i)}$ are induced from $\psi$ and $\varphi$ respectively.

\begin{Thm}[{\cite[Proposition 12.1 and Corollary 12.2]{GLS}}]
The complexes \eqref{eq:bmodres} and \eqref{res3} are exact.
\end{Thm}

By \cite[Theorem 3.16]{M}, we know that $\Ker(\psi^{(i)})$ is isomorphic to $E_{i^*}$ after forgetting the bigrading. 
Our first aim is to prove the following:
\begin{Thm} \label{Ker}
For any $i \in I$, we have $\Ker(\psi^{(i)}) \cong q^{-rh^\vee}t^hE_{i^*}$. In particular, each $E_i$ has the bigraded projective resolution
\begin{equation} \label{PE}
\cdots \to
P_3^{(i)} \to
P_2^{(i)} \to
P_1^{(i)} \to
P_0^{(i)} \to E_i \to 0
\end{equation}  
which extends \eqref{res3} and satisfies $P_{k + 3}^{(i)} = q^{-rh^\vee}t^hP_k^{(i^*)}$ for any $k \in \Z_{\ge 0}$.
\end{Thm} 

A proof of Theorem~\ref{PE} is given in the next subsection.

\begin{Cor} \label{P=I}
For each $i \in I$, we have the following isomorphisms of bigraded $\Pi(\ell)$-modules:
\begin{enumerate}
\item \label{eq:P=I} $P_i(\ell) \cong q^{r(2\ell-h^\vee)}t^{h-2}I_{i^*}(\ell)$,
\item \label{eq:DI=I} $\bD(\bar{I}_i) \cong q^{2d_i-rh^\vee}t^{h-2}\bar{I}_{i^*}$.\qedhere
\end{enumerate}
\end{Cor}
\begin{proof}
Note that we have 
\begin{equation} \label{eq:S=Ker}
q^{2(\ell r -d_i)}S_{i} \cong \Ker\left(E_{i} \xrightarrow{\cdot \ep_i} q^{-2d_i}E_i\right) \subset E_i
\end{equation} 
as bigraded $\Pi(\ell)$-modules.
Combining with Theorem~\ref{Ker}, we have an embedding 
\[
q^{r(2\ell-h^\vee)}t^{h-2} S_{i^*} \hookrightarrow q^{2d_i - rh^\vee}t^{h-2}E_{i^*} \hookrightarrow  P_i. 
\]
Since $P_i$ is isomorphic to $I_{i^*}$ up to bigrading shift by Theorem~\ref{Thm:P=I}, it implies the isomorphism \eqref{eq:P=I}. 
Next, by applying $I_i \otimes_{H_i} (-)$ to \eqref{eq:S=Ker} and recalling \eqref{eq:III}, we obtain $I_i \otimes_{H_i} S_i \cong q^{2(d_i-\ell r)}\bar{I}_i$. 
When $i^* = i$, applying $(-)\otimes_{H_i}S_i$ to the isomorphism \eqref{eq:P=I}, we have
\[ \bD(\bar{I}_i) \cong P_i \otimes_{H_i} S_i \cong q^{r(2\ell-h^\vee)}t^{h-2} I_i \otimes_{H_i}S_i \cong q^{2d_i-rh^\vee}t^{h-2}\bar{I}_{i},\]
which is \eqref{eq:DI=I} in this case. 
When $i \neq i^*$, $\fg$ is necessarily of simply-laced type.
Then, we apply $(-)\otimes_{\kk[\ep]/(\ep^\ell)}\kk$ to the isomorphism \eqref{eq:P=I} and compute similarly, to obtain \eqref{eq:DI=I}.     
\end{proof}

\subsection{$E$-filtrations of $P_i$} 
\label{Ssec:filt}
First, we prepare additional notation. 
For each $i \in I$, we set
$J_i\coloneqq \Pi(1-e_i)\Pi$. 
This is a bigraded two-sided ideal of $\Pi$. 
For $M \in \Cc(\Pi)$ and $i\in I$, let $\sub_i{M}$ (resp.~$\fac_i{M}$) be the largest bigraded submodule (resp.~factor module) of $M$ such that $e_i\sub_i{M}=\sub_i{M}$ (resp.~$e_i\fac_i{M}=\fac_i{M}$). 
In what follows, we endow the localized Grothendieck group $K(\Pi)_{loc}$ (see \S \ref{Ssec:Grot}) with an action of the braid group $B_\fg$ via the isomorphism $K(\Pi)_{loc} \cong \hq$ which identifies the class $[E_i]$ with the element $\alpha_i^\vee$ for each $i \in I$. 
Namely, recalling \eqref{Tcoroot}, we set 
\begin{equation} \label{eq:TE}
T_i [E_j] \seq [E_j] - q^{-d_j}tC_{ji}(q,t)[E_i]
\end{equation}
in $K(\Pi)_{loc}$ for $i,j \in I$.
We can find an analogue of \cite[Proposition 9.4]{GLS} by an easy adaptation of arguments about relationships between idempotent ideals and reflection functors in \cite{BKT,Kul} as follows.

\begin{Lem}\label{lem:reflBraid}
Let $M\in \Cc_{\lf}(\Pi)$. 
\begin{enumerate}
\item \label{lem:reflBraid:sub} 
If $\sub_{i}(M)=0$, we have $J_i \otimes_{\Pi} M \in \Cc_{\lf}(\Pi)$ and $[J_i \otimes_{\Pi} M] = T_i[M]$.
\item \label{lem:reflBraid:fac}
If $\fac_{i}(M)=0$, we have $\hom_{\Pi}(J_i, M) \in \Cc_{\lf}(\Pi)$ and $[\hom_{\Pi}(J_i, M)]=T_i^{-1}[M]$. \qedhere
\end{enumerate}
\end{Lem}
\begin{proof}
We only prove the assertion \eqref{lem:reflBraid:sub} because the assertion \eqref{lem:reflBraid:fac} is dual to \eqref{lem:reflBraid:sub}.
Since $M\in \Cc_{\lf}(\Pi)$, we can write $[M] = \sum_{j\in I} f_j [E_j]$ in $K(\Pi)_{loc}$ with some $f_j \in \Z_{\geq 0}[q^{\pm 1}, t^{\pm 1}]$.
Note that $e_iJ_i$ is the first syzygy of $E'_i$ in $\Cc(\Pi^{\op})$. 
We get the following exact sequence of bigraded $(H_i, \Pi)$-bimodules by applying $E'_i \otimes_{\Pi} (-)$ to \eqref{eq:bmodres} and taking a kernel:
\[
q^{-2d_i}t^2 e_i\Pi \rightarrow \bigoplus_{j\sim i} {}_iH_j\otimes_j e_j\Pi
\rightarrow e_i J_i \rightarrow 0. 
\]
We apply $(-) \otimes_{\Pi} M $ to the above exact sequence to obtain a short exact sequence
\[
0 \rightarrow q^{-2d_i}t^2 e_i M \xrightarrow{\zeta} \bigoplus_{j\sim i} {}_iH_j\otimes_j e_j M
\rightarrow e_i J_i \otimes_{\Pi} M \rightarrow 0.
\]
Here the map $\zeta$ is injective because $\sub_i M= 0$. 
In particular, we have $e_i J_i \otimes_\Pi M \cong \Cok(\zeta)\cong H_i^{\oplus a}$, where 
\[
a = -\sum_{j\sim i} q^{-d_j}tC_{ji}(q, t)f_j -q^{-2d_i}t^2 f_i
\]
by Lemma~\ref{lem:iHjCartan}. 
Combined with the fact $e_j J_i \otimes_\Pi M \cong e_j M$ for $j\neq i$, we obtain $J_i \otimes_\Pi M \in \Cc_{\lf}(\Pi)$ and an equality
\begin{align*}
    [J_i \otimes_\Pi M]&= \sum_{j\neq i}f_j[E_j] - \sum_{j\sim i} q^{-d_j}tC_{ji}(q, t)f_j[E_i] -q^{-2d_i}t^2f_i [E_i]\\
    &=[M]- \sum_{j \in I} q^{-d_j}t C_{ji}(q,t) f_j[E_i]
\end{align*}
in $K(\Pi)_{loc}$.
The right hand side is equal to $T_i[M]$ by \eqref{eq:TE}.
\end{proof}

\begin{Lem}\label{lem:filtJ}
Let $(i_1, \dots, i_l)$ be a reduced expression of $w_0$. 
For any $i \in I$ and $1 \le k \le l$, we have a bigraded $\Pi$-module isomorphism
    \begin{equation*}
    J_{i_{k-1}}\cdots J_{i_1}e_i/ J_{i_k}\cdots J_{i_1}e_i \cong E_{i_k}^{\oplus (\varpi^{\vee}_{i}, T_{i_1}\cdots T_{i_{k-1}}\alpha_{i_k})_{q,t}}.
    \end{equation*}
    \end{Lem}
\begin{proof}
We have an isomorphism
    $J_{i_{k-1}}\cdots J_{i_1} / J_{i_k}\cdots J_{i_1} \cong
        E'_{i_k}\otimes_{\Pi} J_{i_{k-1}} \cdots J_{i_1}$ in $\Cc(\Pi^{\op})$ by a bigraded analogue of an argument in \cite[Proposition 3.8]{M2}. 
        Then, the right module version of Lemma \ref{lem:reflBraid} yields an isomorphism $J_{i_{k-1}}\cdots J_{i_1}e_i/ J_{i_k}\cdots J_{i_1}e_i \cong H_i^{\oplus a}$ in $\Cc(H_i^\op)$ with $a = q^{d_i}t^{-1}(\varpi_{i}, T_{i_1}\cdots T_{i_{k-1}}\alpha_{i_k}^\vee)_{q,t}$ (recall the relation $q^{d_i}t^{-1}(\varpi_i, \alpha_j^\vee)_{q,t} = \delta_{ij}$).
Thus, we have 
\begin{equation} \label{eq:dim1}
\dim_{q,t}(J_{i_{k-1}}\cdots J_{i_1}e_i/J_{i_k}\cdots J_{i_1}e_i) = \frac{1-q^{2\ell r}}{1-q^{2d_i}}q^{d_i}t^{-1}(\varpi_{i}, T_{i_1}\cdots T_{i_{k-1}}\alpha_{i_k}^\vee)_{q,t}.
\end{equation}
On the other hand, by a bigraded analogue of \cite[Lemma 3.10]{M2}, we have an isomorphism $J_{i_{k-1}}\cdots J_{i_1}e_i/ J_{i_k}\cdots J_{i_1}e_i \cong E_{i_k}^{\oplus b}$ in $\Cc(\Pi)$ with some $b \in \Z_{\ge 0}[q^{\pm 1}, t^{\pm 1}]$. 
This yields 
\begin{equation} \label{eq:dim2}
\dim_{q,t} (J_{i_{k-1}}\cdots J_{i_1}e_i/J_{i_k}\cdots J_{i_1}e_i) = b \frac{1-q^{2\ell r}}{1-q^{2d_{i_k}}}. 
\end{equation}
Comparing \eqref{eq:dim1} and \eqref{eq:dim2}, we obtain
\[ b = \frac{1-q^{2d_{i_k}}}{1-q^{2d_i}}q^{d_i}t^{-1}(\varpi_{i}, T_{i_1}\cdots T_{i_{k-1}}\alpha_{i_k}^\vee)_{q,t} 
= (\varpi^{\vee}_{i}, T_{i_1} \cdots T_{i_{k-1}}\alpha_{i_k})_{q,t},\]
where we used the relations $\alpha_{i_k}^\vee = q^{-d_{i_k}}t\alpha_{i_k}/[d_{i_k}]_q$ and  $\varpi_i = [d_i]_q \varpi_i^\vee$.
\end{proof}

\begin{proof}[Proof of Theorem \ref{Ker}]
It is known that $J_{i_l}\cdots J_{i_1}=0$ holds for any reduced expression $(i_1, \ldots, i_l)$ of $w_0$ (cf.~\cite[Theorem 1.2]{FG} and \cite[Theorem 2.31]{M2}).
Therefore, by Lemma~\ref{lem:filtJ}, we have $J_{i_{l-1}}\cdots J_{i_1}e_i \cong 
E_{i_l}^{\oplus a}$ as bigraded (left) $\Pi$-modules, where
\[
a = (\varpi_i^\vee, T_{i_1} \cdots T_{i_{l-1}}\alpha_{i_l})_{q,t} =
(\varpi^{\vee}_{i}, T_{w_0} T_{i_l}^{-1}\alpha_{i_l})_{q,t} =q^{2d_i-rh^\vee} t^{-2+h}\delta_{i^*, i_l}.
\]
Here, the last equality follows from  Theorem~\ref{Thm:Tw_0} and the relation $T_{i_l}^{-1} \alpha_{i_l} = -q^{2d_{i_l}}t^{-2}\alpha_{i_l}$. 
Now we choose a reduced expression $(i_1, \dots, i_l)$ of $w_0$ with $i_l = i^*$. Then $P_i$ contains $J_{i_{l-1}}\cdots J_{i_1}e_i \cong q^{2d_i-rh^\vee}t^{-2+h}E_{i^*}$ as a submodule. 
Since $\Pi$ is self-injective (recall Theorem~\ref{Thm:P=I}), there are no ideals isomorphic to $E_{i^*}$ other than $J_{i_{l-1}}\cdots J_{i_1}e_i$ even if disregarding the bigradings (cf. \cite[Proof of Lemma 2.20]{Miz}).
On the other hand, we know that $\Ker \psi^{(i)}$ in $q^{-2d_i}t^2 P_i$ is isomorphic to $E_{i^*}$ after forgetting the bigradings by \cite[Theorem 3.16]{M}.
Therefore, we have
\[ \Ker \psi^{(i)} \cong q^{-2d_i}t^2 J_{i_{l-1}}\cdots J_{i_1}e_i \cong q^{-rh^\vee}t^{h} E_{i^*}\]
as bigraded $\Pi$-modules.
\end{proof}

We have the following immediate corollary of Lemma \ref{lem:filtJ}, which we use later.
\begin{Cor} \label{Cor:filtJ}
For each $i \in I$, the projective $\Pi$-module $P_i$ is $E$-filtered.
Moreover, we have 
\begin{equation}
[P_i]=\sum_{k=1}^l (\varpi^{\vee}_{i}, T_{i_1} \cdots T_{i_{k-1}}\alpha_{i_k})_{q,t}[E_{i_k}] \label{eq:qtrankPi}
\end{equation}
in $K(\Pi)_{loc}$ for any reduced expression $(i_1, \ldots, i_l)$ of $w_0$.
\end{Cor}
\begin{proof}
By Lemma~\ref{lem:filtJ}, the filtration $0=J_{i_l}\cdots J_{i_1}e_i\subseteq \cdots \subseteq J_{i_k} \cdots J_{i_1}e_i\subseteq \cdots \subseteq \Pi e_i$ gives an $E$-filtration of $P_i$ for any reduced expression $(i_1, \ldots, i_l)$ of $w_0$.
\end{proof}

\subsection{Bigraded dimension of $\bar{I}_i$}
Let $i \in I$.
Applying Corollary~\ref{Cor:ME} to the case $M=P_i$, we have
\begin{equation} \label{eq:PE}
[P_i] = \sum_{j \in I} \dim_{q^{-1}, t^{-1}} (e_i \bar{I}_j) [E_j]
\end{equation}
in $K(\Pi)_{loc}$.
Combining with \eqref{eq:qtrankPi}, we obtain the following.

\begin{Prop} \label{Prop:IT}
Let $(i_1, \ldots, i_l)$ be a reduced expression of $w_0$.
For any $i,j \in I$, we have 
\begin{equation} \label{eq:IT}
\dim_{q^{-1},t^{-1}}(e_i \bar{I}_j) = \sum_{k\colon i_k =j} (\varpi^\vee_i, T_{i_1} \cdots T_{i_{k-1}}\alpha_j)_{q,t}. 
\end{equation}
\end{Prop}
\begin{Cor} \label{Cor:dimIbd}
For any $i,j \in I$, we have 
\[q^{d_j}t^{-1} \dim_{q,t} e_i \bar{I}_j \in \left(q^{d_i}t^{-1}\Z[q,t^{-1}]\right) \cap \left( q^{rh^\vee-d_i}t^{-h+1}\Z[q^{-1},t]\right). \]
\end{Cor}
\begin{proof}
Keep the notation in Proposition~\ref{Prop:IT}.
For any $k$ with $i_k = j$, we have
\[ (\varpi^\vee_i, T_{i_1}\cdots T_{i_{k-1}}\alpha_j)_{q,t} = \frac{q^{d_j}-q^{-d_j}}{q^{d_i}-q^{-d_i}} (T_{i_{k-1}} \cdots T_{i_1} \varpi_i, q^{d_j}t^{-1}\alpha^\vee_j)_{q,t}\]
by Lemma~\ref{lem:eqpair} and the relations $\varpi_i = [d_i]_q \varpi_i^\vee$, $\alpha_j = q^{d_j}t^{-1}[d_j]_q \alpha_j^\vee$.
With Lemma~\ref{Lem:quadrant}, it implies $(\varpi^\vee_i, T_{i_1}\cdots T_{i_{k-1}}\alpha_j)_{q,t} \in q^{d_j -d_i}\Z[q^{-1},t]$.
Therefore, by Proposition~\ref{Prop:IT}, we get 
\[q^{d_j}t^{-1} \dim_{q,t} e_i \bar{I}_j \in q^{d_i}t^{-1}\Z[q,t^{-1}].\]
Combining this with Corollary~\ref{P=I}~\eqref{eq:DI=I}, we obtain the assertion. 
\end{proof}

\subsection{Euler-Poincar\'{e} pairing}  
\label{Ssec:EP}
We consider the following finiteness condition for a pair $(M,N)$ of modules in $\Cc(\Pi)$:
\begin{itemize}
\item[$(\heartsuit)$]
For each $u, v \in \Z$, the extension group $\Ext_{\Pi}^m(q^u t^v M, N)$ vanishes for $m \gg 0$.
\end{itemize}
If $(M,N)$ satisfies the condition $(\heartsuit)$, their \emph{Euler-Poincar\'{e} pairing} 
\begin{align*}
\langle M, N \rangle_{q,t} &\seq \sum_{m=0}^\infty (-1)^m \dim_{q,t} \gext^m_{\Pi}(M,N) \allowdisplaybreaks\\
&= \sum_{m=0}^\infty \sum_{u,v \in \Z} (-1)^m q^u t^v \dim_\kk \Ext^m_{\Pi}(q^u t^v M,N) 
\end{align*}
is well-defined as a formal power series in $q^{\pm 1}, t^{\pm 1}$. 
The following lemma is immediate from the definition and the standard argument using the long exact sequences for $\gext_{\Pi}^{m}(-,-)$'s.

\begin{Lem} \label{Lem:EP}
Let $M, N \in \Cc(\Pi)$.
\begin{enumerate}
\item If $(M,N)$ satisfies $(\heartsuit)$, the pair $(M^{\oplus a}, N^{\oplus b})$ also satisfies $(\heartsuit)$ for any $a, b \in \Z_{\ge 0}[q^{\pm 1}, t^{\pm 1}]$ and we have 
$\langle M^{\oplus a}, N^{\oplus b} \rangle_{q,t} = \bar{a} b \langle M, N \rangle_{q,t}$, where $\ol{a(q,t)} \seq a(q^{-1}, t^{-1})$.
\item Suppose that there is an exact sequence $0 \to M' \to M \to M'' \to 0$ in $\Cc(\Pi)$, and the both pairs $(M', N)$ and $(M'', N)$ satisfy $(\heartsuit)$.
Then the pair $(M,N)$ also satisfies $(\heartsuit)$ and we have $\langle M, N\rangle_{q,t} = \langle M', N \rangle_{q,t} + \langle M'', N \rangle_{q,t}.$
\item Suppose that there is an exact sequence $0 \to N' \to N \to N'' \to 0$ in $\Cc(\Pi)$, and the both pairs $(M, N')$ and $(M, N'')$ satisfy $(\heartsuit)$.
Then the pair $(M,N)$ also satisfies $(\heartsuit)$ and we have $\langle M, N\rangle_{q,t} = \langle M, N' \rangle_{q,t} + \langle M, N'' \rangle_{q,t}.$  \qedhere
\end{enumerate}
\end{Lem}

Using Theorem~\ref{Ker}, we can prove the following proposition. 
This is an advantage of working with the bigrading. If we forget the $t$-degree and work only with the $q$-degree, the analogous statement fails in general (see Example~\ref{Ex:EP} below).

\begin{Prop} \label{Prop:heart}
For any $M,N \in \Cc(\Pi)$, the pair $(M,N)$ satisfies the condition $(\heartsuit)$. Namely, the pairing $\langle M, N \rangle_{q,t}$ only depends on their classes $[M], [N]\in K(\Cc(\Pi))$.
\end{Prop}
\begin{proof}
In view of Lemma~\ref{Lem:EP}, it is enough to show that the pair $(S_i, S_j)$ satisfies the condition $(\heartsuit)$ for any $i, j \in I$.
The simple module $S_i$ has the following ``$E_i$-resolution":   
\[ \cdots \to q^{p(3)} E_i \xrightarrow{\ep_i^{c}} q^{p(2)}E_i \xrightarrow{\ep_i} q^{p(1)}E_i \xrightarrow{\ep_i^{c}} q^{p(0)} E_i \xrightarrow{\ep_i} E_i \to S_i \to 0,\]
where $c \seq \ell r/d_i -1$, and $p \colon \Z_{\ge 0} \to \Z_{\ge 0}$ is a strictly increasing function given by $p(2k) = 2k\ell r + 2d_i$ and $p(2k-1) = 2k\ell r$. 
Using Theorem~\ref{Ker}, let us take the projective resolution of each term $q^{p(k)}E_i$ to obtain a double complex $(P^{(i)}_{u,v})_{u,v \ge 0}$ with $P^{(i)}_{u,v} = q^{p(u)}P^{(i)}_v$, whose total complex gives a bigraded projective $\Pi$-resolution of $S_i$. 
Since $\dim_{q,t} \ghom_{\Pi}(P_v^{(i)}, S_j) \in t^{-s(v)}\Z[q^{\pm 1}]$ with a strictly increasing function $s$, we observe that $\sum_{u,v \ge 0} \dim_{q,t}\ghom_{\Pi}(P_{u,v}^{(i)}, S_j) \in \Z(\!(q^{-1})\!)[\![t^{-1}]\!]$
holds.
This implies the condition $(\heartsuit)$ for $(S_i, S_j)$.
\end{proof}

\begin{Ex} \label{Ex:EP}
Let $\fg$ be of type $\mathrm{C}_2$ ($=\mathrm{B}_2$) and set $I = \{1,2\}$ with $(d_1, d_2) = (1,2)$. 
Taking the bigraded projective resolution of the simple module $S_1$ as in the proof of Proposition~\ref{Prop:heart} above, we can directly compute  
\[
\dim_{q,t} \gext_{\Pi(1)}^m (S_1, S_1) = q^{-2m}\sum_{l=0}^{\lfloor 2m/3 \rfloor}(q^6 t^{-2})^l
\]
for any $m \in \Z_{\ge 0}$.
Therefore we have
\begin{equation} \label{eq:S1S1}
\langle S_1, S_1 \rangle_{q,t} = \sum_{m=0}^{\infty} (-q^{-2})^m \sum_{l=0}^{\lfloor 2m/3 \rfloor}(q^6 t^{-2})^l,
\end{equation}
which gives a well-defined element of $\Z(\!(q^{- 1})\!)[\![t^{-1}]\!]$.
On the other hand, specializing $t$ to $1$ in \eqref{eq:S1S1} does not give a well-defined element of $\Z[\![q^{\pm 1}]\!]$. 
This implies that the Euler-Poincar\'e pairing between $S_1$ and itself is ill-defined if we forget the $t$-grading.    
\end{Ex}

\subsection{Interpretation of $(q,t)$-deformed Cartan matrices}
\label{Ssec:main}
For any $i ,j \in I$, Theorem~\ref{Ker} implies that 
\[
\dim_{q,t} \gext^{m+3}_\Pi (E_i, S_j) = q^{rh^\vee}t^{-h} \dim_{q,t} \gext^m_\Pi (E_{i^*}, S_j)
\]
holds for any $m \in \Z_{\ge 0}$, and we have
\[
\dim_{q,t} \gext^m_\Pi(E_i, S_j) = \begin{cases}
\delta_{ij} & \text{if $m=0$}, \\
(\delta_{ij}-1)q^{d_i}t^{-1}C_{ij}(q,t) & \text{if $m=1$}, \\
\delta_{ij}q^{2d_i}t^{-2} & \text{if $m=2$}.
\end{cases}
\]
Therefore we get
\begin{equation} \label{ES}
\langle E_i, S_j \rangle_{q,t} = \frac{q^{d_i}t^{-1}}{1-(q^{rh^\vee}t^{-h})^2} \left( C_{ij}(q,t) - q^{rh^\vee}t^{-h} C_{i^* j}(q,t)\right) 
\end{equation}
as an element of $\Z[q^{\pm 1}](\!(t^{-1})\!)$.
Let $\nu \seq (\delta_{ij^*})_{i,j \in I}$ be the permutation matrix corresponding to the involution $i \mapsto i^*$. 
Then the equation \eqref{ES} can be expressed in the matrix identity 
\begin{equation} \label{ESmat}
\left(\langle E_i, S_j \rangle_{q,t}\right)_{i,j \in I} = \frac{q^Dt^{-1}(\id-q^{rh^\vee}t^{-h}\nu)}{1-(q^{rh^\vee}t^{-h})^2} C(q,t)
\end{equation} 
in $GL_I\left(\Z[q^{\pm 1}](\!(t^{-1})\!)\right)$.
If the involution $i \mapsto i^*$ is trivial, this identity \eqref{ESmat} simplifies to 
\[
\left(\langle E_i, S_j \rangle_{q,t}\right)_{i,j \in I} = \frac{q^Dt^{-1}}{1+q^{rh^\vee}t^{-h}} C(q,t),
\]
which applies especially to the case of non-simply-laced type.

For any $i,j \in I$, we have
\[
\delta_{ij}  = \langle P_i, S_j \rangle_{q,t}
= \sum_{k \in I} (\dim_{q,t}e_i\bar{I}_k) \langle E_k, S_j \rangle_{q,t}, 
\]
where we used Corollary~\ref{Cor:filtJ}, Lemma~\ref{Lem:EP} and \eqref{eq:PE}.
This can be expressed in an equation of matrices:
\begin{align*}
\id &= \left( \dim_{q,t} e_i\bar{I}_j\right)_{i,j \in I} \left(\langle E_i, S_j \rangle_{q,t}\right)_{i,j \in I} \allowdisplaybreaks\\
&= \frac{1}{1-(q^{rh^\vee}t^{-h})^2} \left( \dim_{q,t}e_i\bar{I}_j\right)_{i,j \in I}  q^{D}t^{-1} (\id - q^{rh^\vee}t^{-h}\nu) C(q,t),
\end{align*}
where the second equality is due to \eqref{ESmat}.
Therefore we have
\[
\tC(q,t) =  \frac{1}{1-(q^{rh^\vee}t^{-h})^2}\left( \dim_{q,t}e_i\bar{I}_j\right)_{i,j \in I} q^{D}t^{-1}(\id - q^{rh^\vee}t^{-h}\nu) 
\]
Comparing the $(i,j)$-entries, we obtain the following formula.
\begin{Thm} \label{Thm:goal}
For any $i, j \in I$, we have
\begin{equation} \label{goal1}
\tC_{ij}(q,t) = \frac{q^{d_j}t^{-1}}{1-(q^{rh^\vee}t^{-h})^2} \left( \dim_{q,t} e_i\bar{I}_j - q^{rh^\vee}t^{-h}\dim_{q,t} e_i\bar{I}_{j^*}\right).
\end{equation}
If the involution $i \mapsto i^*$ is trivial (e.g.~when $\fg$ is of non-simply-laced type), it simplifies to 
\[ \tC_{ij}(q,t) = \frac{q^{d_j}t^{-1}}{1+q^{rh^\vee}t^{-h}} \dim_{q,t} e_i\bar{I}_j.
\]
\end{Thm}

\begin{Cor} \label{Cor:goal}
For any $i,j \in I$, we have
\begin{equation*}
    \dim_{q,t} e_i \bar{I}_j = q^{-d_j}t\sum_{u=0}^{rh^\vee} \sum_{v=0}^{h} \tc_{ij}(u,-v)q^u t^{-v}.
\end{equation*}
\end{Cor}
\begin{proof}
It follows from Theorem~\ref{Thm:goal} and Corollary~\ref{Cor:dimIbd}.
\end{proof}

\begin{Cor} \label{Cor:properties}
The integers $\{ \tc_{ij}(u,v) \}_{i,j \in I, u,v \in \Z}$ satisfy the following properties.
\begin{enumerate}
\item $\tc_{ij}(u,v) = - \tc_{ij^*}(u+rh^\vee, v-h)$ for any $u \ge 0$ and $v \le 0$, 
\item $\tc_{ij}(u,v) \ge 0$ if $0 \le u \le rh^\vee$ and $-h \le v \le 0$,
\item $\tc_{ij}(rh^{\vee}-u, -h-v) = \tc_{ij^{*}}(u,v)$ for any $0 \le u \le rh^\vee$ and $-h \le v \le 0$. \qedhere
\end{enumerate}
\end{Cor}
\begin{proof}
If follows from Theorem~\ref{Thm:goal}, Corollary~\ref{Cor:goal} and Corollary~\ref{P=I}~\eqref{eq:DI=I}.
\end{proof}

As a by-product, we also obtain the following combinatorial formula.

\begin{Prop} \label{Prop:comb}
Let $(i_1, \ldots, i_l)$ be a reduced expression of $w_0$.
We extend it to an infinite sequence $(i_k)_{k \in \Z_{>0}}$ such that $i_{k+l} = i_k^*$ for all $k \in \Z_{>0}$.
Then, for any $i,j \in I$, we have
\[ \tC_{ij}(q,t) = q^{d_j}t^{-1} \sum_{k>0, i_k=j} (\varpi_i^\vee, T_{i_1}^{-1} \cdots T_{i_{k-1}}^{-1}\alpha_j)_{q,t}.\]
\end{Prop} 
\begin{proof}
It follows from Theorem~\ref{Thm:goal}, Proposition~\ref{Prop:IT} and Theorem~\ref{Thm:Tw_0}. 
Note that we have $\overline{(\varpi_i^\vee, T_{i_1} \cdots T_{i_{k-1}} \alpha_j)_{q,t}} = (\varpi_i^\vee, T_{i_1}^{-1} \cdots T_{i_{k-1}}^{-1}\alpha_j)_{q,t}$ for any $k \in \Z_{>0}$ by~\eqref{Troot}, where $\overline{a(q,t)}\seq a(q^{-1}, t^{-1})$.
\end{proof}
\begin{Rem}
When $\fg$ is of simply-laced type and the reduced expression $(i_1, \ldots, i_l)$ is adapted to a Dynkin quiver, Proposition~\ref{Prop:comb} recovers Hernandez-Leclerc's formula \cite[Proposition 2.1]{HL15}.
For the other case, it seems new.
\end{Rem}
\section{Remarks on $q$-version and projective limit}
\label{Sec:Rem}

In this complimentary section, we switch to consider the graded version of $\Pi(\ell)$ and $\tPi$.
In \S \ref{Ssec:qC}, we summarize the $q$-version of our results obtained so far.
In \S \ref{Ssec:HL}, we explain in detail how to identify graded $\tPi$-modules with modules over the Jacobian algebra studied by Hernandez-Leclerc~\cite{HL}. 
In \S \ref{Ssec:projlim}, we discuss the projective limit of $\Pi(\ell)$'s, which we need in the next section.

\subsection{Change of conventions}
In the remaining part of this paper (Sections~\ref{Sec:Rem} \& \ref{Sec:K} and Appendix~\ref{Apx}), we regard $\tPi$ as a graded $\kk$-algebra with respect to the degree map $\deg_1 \seq \mathrm{pr}_1 \circ \deg$, where $\mathrm{pr}_1 \colon \Z^2 \to \Z$ is the projection of the first factor.
Explicitly, it is given by
\[ \deg_1(\alpha_{ij}) \seq b_{ij}= - \max(d_i, d_j), \qquad \deg_1(\ep_i) \seq b_{ii} = 2d_i.\]

We again write $q$ for the upward grading shift functor for graded $\kk$-vector spaces.
For a graded $\kk$-vector space $V$, 
its graded dimension $\dim_q V$ and restricted dual $\bD(V)$ are defined in the analogous way as in \S \ref{Ssec:bgv}. 
For a graded $\kk$-algebra $A$, we write $\Hom_A(M,N)$ for the space of homogeneous $A$-homomorphisms between graded $A$-modules $M$ and $N$. 
We define $\ghom_{A}(M,N) \seq \bigoplus_{u \in \Z} \Hom_A(q^uM,N)$ as graded $\kk$-vector space.
The same convention applies to $\Ext_A^m(M,N)$ and $\gext_A^m(M,N)$ as well. 

In the sequel, we work only on the category of graded $\tPi$-modules, instead of bigraded $\tPi$-modules.
To simplify the notation, for a bigraded $\tPi$-module $M$, we keep the same symbol $M$ to denote the graded $\tPi$-module obtained from $M$ by forgetting the ``$t$-degree''.
Namely, we regard $M$ as the graded $\tPi$-module whose $u$-th graded piece is given by $M_u \seq \bigoplus_{v \in \Z} M_{u,v}$ for each $u \in \Z$. 
In particular, even if $M$ and $N$ are bigraded $\tPi$-modules, we switch to use the symbol $\Hom_{\tPi}(M,N)$ to denote the space of graded $\tPi$-homomorphisms, rather than bigraded ones. 
The same convention applies to $\Ext_{\tPi}^m(M,N)$ and others.

\subsection{Inverse of $q$-deformed Cartan matrix}
\label{Ssec:qC}
We set $\tC(q) \seq \tC(q,1)$.
This is the inverse of the $q$-deformed Cartan matrix $C(q) = C(q,1)$. 
Its $(i,j)$-entry $\tC_{ij}(q) = \tC_{ij}(q,1)$ is expanded at $q=0$ as
\[ \tC_{ij}(q) = \sum_{u \in \Z}\tc_{ij}(u)q^u \in \Z[\![q]\!], \qquad \text{where } \tc_{ij}(u) \seq \sum_{v \in \Z}\tc_{ij}(u,v).\]
Here we list some properties of $\tC(q)$ for future reference.
They immediately follow from Corollaries~\ref{Cor:goal} \& \ref{Cor:properties} together with Lemma~\ref{Lem:tc}.
See \cite[Corollary 4.10]{FO} and \cite[\S 6.6]{GW} for alternative proofs of Proposition~\ref{Prop:tc}.

\begin{Prop} \label{Prop:dI} 
For each $i, j \in I$, we have
\begin{equation} \label{eq:dI}
 \dim_q e_i\bar{I}_j = q^{-d_j}\sum_{u=0}^{rh^\vee} \tc_{ij}(u)q^{u}.
\end{equation}
In particular, the space $e_i \bar{I}_i$ is non-negatively graded.
\end{Prop}

\begin{Prop} \label{Prop:tc}
The expansion coefficients $\{\tc_{ij}(u)\}_{i,j \in I, u \in \Z}$ satisfy the following properties:
\begin{enumerate}
\item $\tc_{ij}(u+rh^\vee) = - \tc_{ij^*}(u)$ for $u \ge 0$.
\item $\tc_{ij}(u) \ge 0$ for $0 \le u \le rh^\vee$.
\item $\tc_{ij}(rh^\vee-u) = \tc_{ij^*}(u)$ for $0 \le u \le rh^\vee$.
\item $\tc_{ij}(u)=0$ if $|u-krh^\vee| \le d_i - \delta_{ij}$ for some $k \in \Z_{\ge 0}$. \qedhere
\end{enumerate}
\end{Prop}

\subsection{Comparison with Jacobian algebras}
\label{Ssec:HL}
As remarked in \cite[\S1.7.1]{GLS}, the definition of the algebra $\tPi$ is inspired in part by \cite{HL}.
In particular, the category of graded $\tPi$-modules can be identified with the category of (non-graded) modules over the Jacobian algebra $J_{\Gamma, W}$ associated with a certain quiver $\Gamma$ with potential $W$ studied in \cite{HL}. 
In this subsection, we explain this identification in detail for completeness.

Following \cite{HL}, let us consider an infinite quiver $\Gamma = (\Gamma_0, \Gamma_1, \se, \tl)$ given as follows:
\begin{gather*}
\Gamma_0 = I \times \Z,  \quad
\Gamma_1 = \{ \alpha_{ij}(p) \mid i, j \in I, i \sim j, p \in \Z \} \cup \{\ep_i(p) \mid i \in I, p \in \Z \}, \\
\se(\alpha_{ij}(p)) = (j,p), \quad \tl(\alpha_{ij}(p)) = (i, p +b_{ij}), \quad
\se(\ep_i(p)) = (i,p), \quad \tl(\ep_i(p)) = (i, p+2d_i).  
\end{gather*} 
Note that the quiver $\Gamma$ consists of two mutually isomorphic connected components. 
In Figure~\ref{Fig:Gamma}, a connected component of $\Gamma$ is depicted in types $\mathrm{A}_5, \mathrm{B_3}, \mathrm{C}_3$ and $\mathrm{D}_4$.

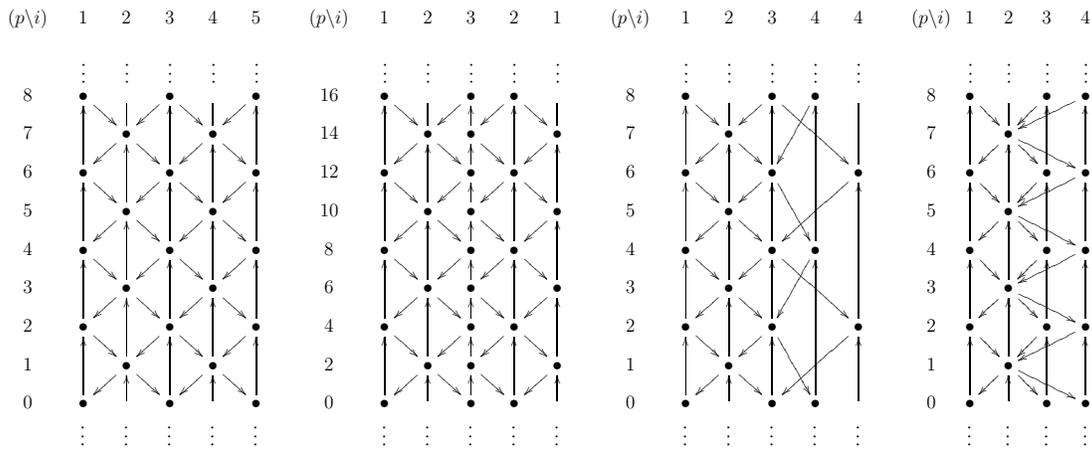
\begin{figure}[htbp] 
\[
\scalebox{0.6}{
\xymatrix@C=5mm@!R=0.1mm{
(p \backslash i) & 1 & 2 &3 & 4 & 5\\
& \text{\rotatebox{270}{$\cdots$}} & \text{\rotatebox{270}{$\cdots$}} & \text{\rotatebox{270}{$\cdots$}} & \text{\rotatebox{270}{$\cdots$}} & \text{\rotatebox{270}{$\cdots$}}\\
8& \bullet \ar@{->}[dr] & & \bullet \ar@{->}[dr] \ar@{->}[dl]& & \bullet \ar@{->}[dl] \\
7 & & \bullet  \ar@{->}[dl] \ar@{->}[dr] \ar@{-}[u] && \bullet \ar@{->}[dl] \ar@{->}[dr] \ar@{-}[u]&\\
6 & \bullet \ar@{->}[uu] \ar@{->}[dr] && \bullet  \ar@{->}[uu] \ar@{->}[dr] \ar@{->}[dl]&& \bullet \ar@{->}[dl] \ar@{->}[uu]\\
5 & & \bullet \ar@{->}[uu] \ar@{->}[dl] \ar@{->}[dr] && \bullet \ar@{->}[uu] \ar@{->}[dl] \ar@{->}[dr]&\\
4 & \bullet \ar@{->}[uu] \ar@{->}[dr] && \bullet  \ar@{->}[uu] \ar@{->}[dr] \ar@{->}[dl]&& \bullet \ar@{->}[dl] \ar@{->}[uu]\\
3 & & \bullet \ar@{->}[uu] \ar@{->}[dl] \ar@{->}[dr] && \bullet \ar@{->}[uu] \ar@{->}[dl] \ar@{->}[dr]&\\
2 & \bullet \ar@{->}[uu] \ar@{->}[dr] && \bullet  \ar@{->}[uu] \ar@{->}[dr] \ar@{->}[dl]&& \bullet \ar@{->}[dl] \ar@{->}[uu]\\
1 & & \bullet \ar@{->}[uu] \ar@{->}[dl] \ar@{->}[dr] & & \bullet \ar@{->}[uu] \ar@{->}[dl] \ar@{->}[dr]& \\
0 & \bullet \ar@{->}[uu] & \ar@{->}[u] & \bullet  \ar@{->}[uu]& \ar@{->}[u] & \bullet \ar@{->}[uu]\\
 &  \text{\rotatebox{90}{$\cdots$}} & \text{\rotatebox{90}{$\cdots$}} & \text{\rotatebox{90}{$\cdots$}} & \text{\rotatebox{90}{$\cdots$}} &  \text{\rotatebox{90}{$\cdots$}} 
}}
\quad
\scalebox{0.6}{
\xymatrix@C=5mm@!R=0.1mm{
(p \backslash i) & 1 & 2 &3 & 2 & 1\\
& \text{\rotatebox{270}{$\cdots$}} & \text{\rotatebox{270}{$\cdots$}} & \text{\rotatebox{270}{$\cdots$}} & \text{\rotatebox{270}{$\cdots$}}& \text{\rotatebox{270}{$\cdots$}} \\
16& \bullet \ar@{->}[dr] && \bullet  \ar@{->}[dl] & \bullet  \ar@{->}[dr] \ar@{->}[dl] &  \\
14& & \bullet  \ar@{->}[dl] \ar@{->}[dr] \ar@{-}[u] &\bullet \ar@{->}[u] \ar@{->}[dr] && \bullet \ar@{->}[dl] \ar@{-}[u]\\
12& \bullet \ar@{->}[uu] \ar@{->}[dr] && \bullet \ar@{->}[u] \ar@{->}[dl] & \bullet \ar@{->}[uu] \ar@{->}[dr] \ar@{->}[dl] &  \\
10 & & \bullet \ar@{->}[uu] \ar@{->}[dl] \ar@{->}[dr] &\bullet \ar@{->}[u] \ar@{->}[dr] && \bullet  \ar@{->}[uu]  \ar@{->}[dl] \\
8& \bullet \ar@{->}[uu] \ar@{->}[dr] && \bullet \ar@{->}[u] \ar@{->}[dl] & \bullet \ar@{->}[uu] \ar@{->}[dr] \ar@{->}[dl] &  \\
6 & & \bullet \ar@{->}[uu] \ar@{->}[dl] \ar@{->}[dr] &\bullet \ar@{->}[u] \ar@{->}[dr] && \bullet  \ar@{->}[uu]  \ar@{->}[dl] \\
4& \bullet \ar@{->}[uu] \ar@{->}[dr] && \bullet \ar@{->}[u] \ar@{->}[dl] & \bullet \ar@{->}[uu] \ar@{->}[dr] \ar@{->}[dl] &  \\
2& & \bullet \ar@{->}[uu] \ar@{->}[dl] \ar@{->}[dr] &\bullet \ar@{->}[u] \ar@{->}[dr] & & \bullet  \ar@{->}[uu] \ar@{->}[dl] \\
0& \bullet \ar@{->}[uu] & \ar@{->}[u]& \bullet \ar@{->}[u] & \bullet \ar@{->}[uu] &\ar@{->}[u]\\
& \text{\rotatebox{90}{$\cdots$}} & \text{\rotatebox{90}{$\cdots$}} & \text{\rotatebox{90}{$\cdots$}} &\text{\rotatebox{90}{$\cdots$}} & \text{\rotatebox{90}{$\cdots$}}
}}
\quad
\scalebox{0.60}{
\xymatrix@C=5mm@!R=0.1mm{
(p \backslash i) & 1 & 2 &3 & 4 & 4 \\
&\text{\rotatebox{270}{$\cdots$}}&\text{\rotatebox{270}{$\cdots$}}&\text{\rotatebox{270}{$\cdots$}}&\text{\rotatebox{270}{$\cdots$}}& \text{\rotatebox{270}{$\cdots$}}\\
8& \bullet \ar@{->}[dr]& & \bullet \ar@{->}[ddrr]\ar@{->}[dl]& \bullet \ar@{->}[ddl]& \\ 
7& & \bullet  \ar@{->}[dl] \ar@{-}[u] \ar@{->}[dr]& & & \\
6& \bullet \ar@{->}[uu] \ar@{->}[dr]& & \bullet \ar@{->}[uu] \ar@{->}[ddr] \ar@{->}[dl]& & \bullet \ar@{->}[ddll] \ar@{-}[uu] \\ 
5& & \bullet \ar@{->}[uu] \ar@{->}[dl] \ar@{->}[dr]& & & \\ 
4& \bullet \ar@{->}[uu] \ar@{->}[dr]& & \bullet \ar@{->}[uu] \ar@{->}[ddrr]\ar@{->}[dl]& \bullet \ar@{->}[uuuu]\ar@{->}[ddl]& \\ 
3& & \bullet \ar@{->}[uu] \ar@{->}[dl] \ar@{->}[dr]& & & \\
2& \bullet \ar@{->}[uu] \ar@{->}[dr]& & \bullet \ar@{->}[uu] \ar@{->}[dl] \ar@{->}[ddr]& & \bullet \ar@{->}[uuuu] \ar@{->}[ddll]\\ 
1& & \bullet \ar@{->}[uu] \ar@{->}[dr] \ar@{->}[dl] & & & \\ 
0& \bullet \ar@{->}[uu]&\ar@{->}[u]  & \bullet \ar@{->}[uu]& \bullet \ar@{->}[uuuu]& \ar@{->}[uu] \\
& \text{\rotatebox{90}{$\cdots$}} &\text{\rotatebox{90}{$\cdots$}}& \text{\rotatebox{90}{$\cdots$}} &\text{\rotatebox{90}{$\cdots$}} & \text{\rotatebox{90}{$\cdots$}}
}}
\quad
\scalebox{0.60}{
\xymatrix@!C=5mm@!R=0.1mm{
(p \backslash i) & 1 & 2 &3 & 4 \\
& \text{\rotatebox{270}{$\cdots$}} & \text{\rotatebox{270}{$\cdots$}} &\text{\rotatebox{270}{$\cdots$}} & \text{\rotatebox{270}{$\cdots$}} \\
8& \bullet \ar@{->}[dr] && \bullet  \ar@{->}[dl]& \bullet \ar@{->}[dll] \\
7 & & \bullet  \ar@{->}[dl] \ar@{->}[dr] \ar@{->}[drr] \ar@{-}[u]&& \\
6 & \bullet \ar@{->}[uu] \ar@{->}[dr] && \bullet  \ar@{->}[uu]  \ar@{->}[dl]&  \bullet \ar@{->}[uu] \ar@{->}[dll] \\
5 & & \bullet \ar@{->}[uu] \ar@{->}[dl] \ar@{->}[dr]\ar@{->}[drr] &&\\
4 & \bullet \ar@{->}[uu] \ar@{->}[dr] && \bullet  \ar@{->}[uu]  \ar@{->}[dl]& \bullet \ar@{->}[uu] \ar@{->}[dll] \\
3 & & \bullet \ar@{->}[uu] \ar@{->}[dl] \ar@{->}[dr] \ar@{->}[drr]&& \\
2 & \bullet \ar@{->}[uu] \ar@{->}[dr] && \bullet  \ar@{->}[uu] \ar@{->}[dl]&  \bullet \ar@{->}[uu] \ar@{->}[dll] \\
1 & & \bullet \ar@{->}[uu] \ar@{->}[dl] \ar@{->}[dr] \ar@{->}[drr] && \\
0 & \bullet \ar@{->}[uu] &\ar@{->}[u]& \bullet  \ar@{->}[uu]&  \bullet \ar@{->}[uu] \\
& \text{\rotatebox{90}{$\cdots$}}&\text{\rotatebox{90}{$\cdots$}} &\text{\rotatebox{90}{$\cdots$}}&\text{\rotatebox{90}{$\cdots$}}}}
\]
\caption{Component of $\Gamma$ in types $\mathrm{A}_5, \mathrm{B_3}, \mathrm{C}_3$ and $\mathrm{D}_4$ (from left to right)} \label{Fig:Gamma}
\end{figure}

Recall we have chosen a collection of signs $\{\omega_{ij}\}_{i \sim j}$ (or equivalently, an orientation $\Omega$) in \S\ref{GPAdef}.
Let $W$ be the potential given by
\[ 
W \seq \sum_{p} \sum_{i,j \in I; i \sim j} \omega_{ij} \ep_i(p-2d_i) \ep_i(p-4d_i) \cdots \ep_i(p+2b_{ij})\alpha_{ij}(p+b_{ij}) \alpha_{ji}(p)
\] 
and $J_{\Gamma, W}$ the Jacobian algebra associated with $(\Gamma, W)$, i.e.,~$J_{\Gamma, W} \seq \kk \Gamma / \langle \partial W\rangle$, where $\langle \partial W\rangle$ denote the two-sided ideal generated by all the cyclic derivations of $W$ (see \cite{DWZ1}). 

\begin{Rem} \label{Rem:potential}
In \cite{HL}, a slightly different potential $W'$ was considered instead of $W$.
It is given by
\[
W' \seq \sum_{p} \sum_{i,j \in I; i \sim j} \ep_i(p-2d_i) \ep_i(p-4d_i) \cdots \ep_i(p+2b_{ij})\alpha_{ij}(p+b_{ij}) \alpha_{ji}(p).
\]
However the difference between $W$ and $W'$ is not essential. 
Indeed, the potential $-W$ is obtained from $W'$ via an explicit automorphism of $\kk\Gamma$ given as follows.
Let $\xi \colon I \to \Z$ be a function satisfying the condition: $\xi_i = \xi_j + \omega_{ij}b_{ij}$ if $i \sim j$, where we write $\xi_i \seq \xi(i)$ for simplicity. (When $\fg$ is of simply-laced type, such a function $\xi$ is the sane as a height function of the Dynkin quiver $(I, \Omega)$ appearing in \cite[\S 2]{HL15}.)
Then we define the automorphism $\phi_{\xi}$ of the path algebra $\kk\Gamma$ by the assignment
\[ \ep_i(p) \mapsto \ep_i(p), \qquad \alpha_{ij}(p) \mapsto (-1)^{\lfloor (\xi_j -p)/2b_{ij} \rfloor} \alpha_{ij}(p) \]
for any $p \in \Z$ and $i,j \in I$ with $i \sim j$. 
Since
\begin{align*}
\left\lfloor \frac{\xi_j-p}{2b_{ij}} \right\rfloor + \left\lfloor \frac{\xi_i-(p - b_{ij})}{2b_{ij}} \right\rfloor 
&=  \left\lfloor \frac{\xi_j-p}{2b_{ij}} \right\rfloor + \left\lfloor \frac{\xi_j-p}{2b_{ij}} + \frac{\xi_i - \xi_j + b_{ij}}{2b_{ij}} \right\rfloor  \allowdisplaybreaks \\
&= \left\lfloor \frac{\xi_j-p}{2b_{ij}} \right\rfloor + \left\lfloor \frac{\xi_j-p}{2b_{ij}} + \frac{\omega_{ij} + 1}{2} \right\rfloor \allowdisplaybreaks \\
& \equiv \frac{\omega_{ij} + 1}{2} \pmod 2,
\end{align*}
we have 
\[ \phi_{\xi} (\alpha_{ij}(p+b_{ij}) \alpha_{ji}(p)) = - \omega_{ij} \alpha_{ij}(p+b_{ij}) \alpha_{ji}(p)\]
for any $p \in \Z$ and $i,j \in I$ with $i \sim j$. 
Thus we obtain $\phi_{\xi}(W') = -W$. 
In particular, the automorphism $\phi_\xi$ induces the isomorphism 
\[J_{\Gamma, W'} \cong J_{\Gamma, -W} = J_{\Gamma, W}. \]
In this paper, we work with the potential $W$ rather than $W'$ because it matches with the definition of $\tPi$ in \S\ref{GPAdef} (see the proof of Proposition~\ref{Prop:catisom} below).  
\end{Rem}
In what follows, for an algebra (resp.~a graded algebra) $A$, we denote by $A\umod$ (resp.~$A\gmod$) the category of all the $A$-modules (resp.~graded $A$-modules).
We naturally identify a graded module $M$ over the path algebra $\kk \tQ$ with a graded representation of $\tQ$, which consists of an $I$-tuple of graded $\kk$-vector spaces $(e_iM)_{i \in I}$ together with linear maps
\[ M(\ep_i) \in \Hom_{\kk} (q^{2d_i}e_i M, e_i M), \qquad M(\alpha_{ij}) \in \Hom_{\kk}(q^{b_{ij}}e_j M, e_i M) \] 
for each $i,j \in I$ with $i \sim j$.

For a graded $\kk \tQ$-module $M$, we associate the representation $\Phi(M)$ of $\Gamma$ over $\kk$ given by $\Phi(M)_{(i,p)} \seq e_i M_p$ and
\[
\Phi(M)(\ep_i(p)) \seq M(\ep_i)|_{e_iM_p}, \qquad \Phi(M)(\alpha_{ij}(p)) \seq M(\alpha_{ij})|_{e_j M_p}
\]
for any $p \in \Z$ and $i,j \in I$ with $j \sim i$.
This assignment $M \mapsto \Phi(M)$ defines an $\kk$-linear functor $\Phi \colon \kk\tQ \gmod \to \kk \Gamma \umod$, which is an isomorphism of categories.

\begin{Prop} \label{Prop:catisom}
Under the isomorphism $\Phi \colon \kk\tQ \gmod \to \kk \Gamma \umod$, the category $\tPi \gmod$ corresponds to the category $J_{\Gamma, W} \umod$.
Therefore there is an isomorphism of categories 
\[\tPi \gmod \cong J_{\Gamma, W} \umod.\]
\end{Prop}
\begin{proof}
The category $J_{\Gamma, W} \umod$ is identical to the full subcategory of $\kk \Gamma \umod$ on which the cyclic derivations $\partial_{\alpha_{ji}(p)}W$ and $\partial_{\ep_i(p)} W$ vanish for any $p \in \Z$ and $i,j \in I$ with $i \sim j$. 
Under the above isomorphism $\kk \Gamma \umod \cong \kk \tQ \gmod$, the actions of the elements $\partial_{\alpha_{ji}(p-b_{ij})}W$ and $\partial_{\ep_i(p-2d_i)} W$ correspond to the restrictions to the $p$-th graded piece of the actions of the elements
\[ 
\omega_{ij}\ep_j^{-c_{ji}} \alpha_{ji} + \omega_{ji} \alpha_{ji} \ep_i^{-c_{ji}} \quad \text{and} \quad \sum_{j\sim i}\sum_{k + l = -c_{ij}-1}\omega_{ij} \ep_i^k \alpha_{ij} \alpha_{ji} \ep_i^{l} 
\]
respectively. 
Therefore, under the isomorphism, the relation $\partial_{\alpha_{ji}(p)}W =0$ corresponds to the relation (R1), and the relation $\partial_{\ep_i(p)} W =0$ corresponds to the relation (R2).
This completes the proof.
\end{proof}

\subsection{Projective limit}
\label{Ssec:projlim}
In this subsection, we briefly discuss the projective limit of the graded $\kk$-algebras $\Pi(\ell)$, which we use in the next section.
Taking the projective limit in the category of graded $\kk$-algebras, we define
\[\Pi(\infty) \seq \varprojlim_{\ell} \Pi(\ell) = \varprojlim_{\ell} \tPi/\ep^{\ell}\tPi.\]
By construction, we have the canonical homomorphism of graded $\kk$-algebras $\tPi \to \Pi(\infty)$, whose kernel is $\bigcap_{\ell > 0} \ep^\ell \tPi$.
For each $\ell \in \Z_{>0}$, we have $\Pi(\infty)/\ep^{\ell}\Pi(\infty) \cong \tPi/\ep^{\ell}\tPi\cong \Pi(\ell)$.
We consider the projective module $P_i(\infty) \seq \Pi(\infty)e_i$ and the injective module $I_i(\infty) \seq \bD(P_i(\infty))$ for each $i \in I$. 
Note that we have $I_i(\infty) = \bigcup_{\ell \in \Z_{>0}}I_i(\ell)$ and hence it is not finitely generated over $\Pi(\infty)$. 
Let $\mathfrak{a} \subset \Pi(\infty)$ denote the two-sided ideal generated by $\{ \alpha_{ij} \}_{i \sim j}$. 
\begin{Prop} \label{Prop:Piinfty} The followings hold.
\begin{enumerate}
\item \label{Prop:Piinfty:lf} For any $i, j \in I$, the subspace $e_i \Pi(\infty) e_j$ is graded free of finite rank over $\kk[\ep_i]$.
\item \label{Prop:Piinfty:ep} The center of $\Pi(\infty)$ contains $\kk[\ep]$ and $\Pi(\infty)$ is graded free of finite rank as a $\kk[\ep]$-module.
In particular, we have $\dim_q \Pi(\infty) \in \Z(\!(q)\!)$. 
\item \label{Prop:Piinfty:sj} For each $k \in \Z$, we have $\Pi(\infty)_k = \Pi(\ell)_k$ for $\ell \gg 0$.   
In particular, the canonical homomorphism $\tPi \to \Pi(\infty)$ is surjective. 
\item \label{Prop:Piinfty:nilp} The ideal $\mathfrak{a}$ is nilpotent. Indeed, we have $\mathfrak{a}^{n(h+1)} = 0$.\qedhere
\end{enumerate}
\end{Prop}
\begin{proof}
Since $e_i \Pi(\infty) e_j = \varprojlim e_i \Pi(\ell) e_j$, \eqref{Prop:Piinfty:lf} follows from Theorem~\ref{Thm:fd}.
Since $\kk[\ep_i]$ is graded free of finite rank over $\kk[\ep]$, \eqref{Prop:Piinfty:ep} follows from \eqref{Prop:Piinfty:lf}. 
\eqref{Prop:Piinfty:sj} is immediate from \eqref{Prop:Piinfty:ep}.
\eqref{Prop:Piinfty:nilp} follows from Theorem~\ref{Thm:fd}~\eqref{Thm:fd:nilp}.
\end{proof}

\begin{Prop}
Assume that $\fg$ is of simply-laced type. 
Then the canonical homomorphism $\tPi \to \Pi(\infty)$ is an isomorphism.
Moreover, we have $\tPi \cong \Pi(1)\otimes_\kk \kk[\ep]$ as graded $\kk$-algebras. 
\end{Prop}
One can expect that $\tPi$ is isomorphic to $\Pi(\infty)$ for general $\fg$, but we could not find a proof.
\begin{proof}
Note that the algebra $\Pi(1)$ is the same as the usual preprojective algebra and hence we have a graded $\kk$-algebra homomorphism $\Pi(1) \to \tPi$ whose image is identical to the subalgebra $A$ of $\tPi$ generated by $\{e_i\}_{i \in I} \cup \{\alpha_{ij} \}_{i \sim j}$.
In particular, $A$ is finite-dimensional.  
Since the element $\ep$ is central, we have $\tPi = \sum_{\ell=0}^{\infty} \ep^\ell A$.
Thus the gradation of $\tPi$ is bounded from below,
which implies that $\Ker(\tPi \to \Pi(\infty)) = \bigcap_{\ell > 0} \ep^\ell \tPi =0$.     
Combined with Proposition~\ref{Prop:Piinfty}~\eqref{Prop:Piinfty:sj}, we obtain the former assertion.
Now the latter assertion follows from Proposition~\ref{Prop:Piinfty}~\eqref{Prop:Piinfty:ep}.   
\end{proof}

In what follows, we identify the category $\Pi(\infty) \gmod$ with a full subcategory of $\tPi\gmod$ via the canonical surjection $\tPi \to \Pi(\infty)$. 
\begin{Prop} \label{Prop:bd}
The category $\Pi(\infty)\gmod$ is identical to the full subcategory of $\tPi\gmod$ consisting of graded $\tPi$-modules $M$ satisfying the following property: for each homogeneous element $y \in M$, the gradation of the submodule $\tPi y \subset M$ is bounded from below.  
\end{Prop}
\begin{proof}
Let $\mathcal{B} \subset \tPi\gmod$ be the full subcategory in question, i.e., $\mathcal{B}$ consists of all the graded $\tPi$-modules $M$ such that $\tPi y \subset M$ is bounded from below for any homogeneous $y \in M$. 
The inclusion $\Pi(\infty)\gmod \subset \mathcal{B}$ follows from Proposition~\ref{Prop:Piinfty}~\eqref{Prop:Piinfty:ep}.
To see the opposite inclusion, it is enough to show that $xM = 0$ holds for any $M \in \mathcal{B}$ and $x \in \Ker(\tPi \to \Pi(\infty)) = \bigcap_{\ell > 0} \ep^\ell \tPi$. 
We may assume that $x$ is homogeneous. 
For any $\ell > 0$, we can write $x=\ep^\ell x_\ell$ for some $x_\ell \in \tPi$ such that $\deg_1(x_\ell) = \deg_1(x)-2\ell r$.
For any $y \in M$, we have $x_\ell y =0$ for $\ell \gg 0$ by the assumption $M \in \mathcal{B}$.
Therefore we have $xy = \ep^\ell (x_\ell y)=0$. 
This completes the proof.
\end{proof}

\begin{Cor} \label{Cor:bd}
The category $\Pi(\infty) \gmod$ is a Serre subcategory of $\tPi\gmod$, i.e., $\Pi(\infty) \gmod$ is closed under taking subobjects, quotients and extensions.
In particular, we have the natural isomorphism $\Ext^1_{\Pi(\infty)}(M,N) \cong \Ext^1_{\tPi}(M,N)$ for any $M,N \in \Pi(\infty) \gmod$.
\end{Cor}

\section{Application to generic kernels}
\label{Sec:K}
In this section, we discuss the generic kernels corresponding to the Kirillov-Reshetikhin (KR) modules introduced by Hernandez-Leclerc~\cite{HL}.
Since they yield the geometric $q$-character formulas \cite[Theorem 4.8]{HL} (see Remark~\ref{Rem:qch} below), one can think of them as an additive counterpart of KR modules in the context of the categorifications of cluster algebras.  
In \S \ref{Ssec:K}, we introduce the generic kernels as certain graded $\tPi$-modules, and explain that our definition is equivalent to the one in~\cite{HL}.  
We compute all the first extension groups among them explicitly in \S \ref{Ssec:Ext}, and compare the results with the conjectural denominator formula of the normalized $R$-matrices due to \cite{FO} in \S \ref{Ssec:conj}. 
Computations for a few exceptional cases are postponed in Appendix~\ref{Apx}. 

\subsection{Generic kernels}
\label{Ssec:K}
For each $i \in I$ and $k \in \Z_{>0}$, we define the graded $\tPi$-module $K^{(i)}_{k}$ by 
\[ K^{(i)}_{k} \seq q^{kd_i}\bD((\tPi/\tPi\ep_i^k)e_i).\]
As special cases, we have $K^{(i)}_{1} \cong q^{d_i}\bar{I}_i$ and $K^{(i)}_{\ell r/d_i} \cong q^{\ell r}I_i(\ell)$ for $\ell \in \Z_{>0}$.
From the definition, the module $K^{(i)}_{k}$ fits into the following short exact sequence:
\begin{equation}\label{sesK}
0 \to K^{(i)}_{k} \to q^{kd_i}I_i(\infty) \xrightarrow{\cdot \ep_i^k} q^{-kd_i}I_i(\infty) \to 0.
\end{equation}
Here the surjectivity of $\cdot \ep_i^k$ follows from Proposition~\ref{Prop:Piinfty}~\eqref{Prop:Piinfty:lf}.
The modules $K^{(i)}_{k}$ are referred to as \emph{generic kernels} after the following fact (see also Remark~\ref{Rem:K_HL} below).

\begin{Prop} \label{Prop:generic}
For each $i \in I$ and $k \in \Z_{>0}$, the set of homomorphisms $f \colon q^{2kd_i}I_i(\infty) \to I_i(\infty)$ satisfying $\Ker(f) = \Ker(\cdot \ep_i^k)$ is Zariski dense in the affine space $\Hom_{\tPi}(q^{2kd_i}I_i(\infty), I_i(\infty))$. 
\end{Prop}

\begin{proof}
To simplify the notation, we set $\Pi \seq \Pi(\infty)$ and $I_i \seq I_i(\infty)$ in this proof.
We have the natural isomorphism
\begin{equation} \label{tr}
\Hom_{\Pi}(q^{2kd_i}I_i, I_i) \xrightarrow{\simeq} \Hom_{\Pi}(\Pi e_i, q^{-2kd_i}\Pi e_i) \xrightarrow{\simeq} (e_i \Pi e_i)_{2kd_i}, 
\end{equation}
which transforms $f \in \Hom_{\Pi}(q^{2kd_i}I_i, I_i)$ into ${}^{\mathtt{t}}\!f(e_i) \in (e_i \Pi e_i)_{2kd_i}$.
Note that the homomorphism $\cdot\ep_i^k \colon q^{2kd_i}I_i \to I_i$ simply corresponds to the element $\ep_i^k \in (e_i \Pi e_i)_{2kd_i}$ under the isomorphism \eqref{tr}.
The group $\Aut_{\Pi}(I_i)$ naturally acts on the space $\Hom_{\Pi}(q^{2kd_i}I_i, I_i)$ from the left. 
By the isomorphism \eqref{tr}, this action is translated into the natural right action of the group $(e_i \Pi e_i)^\times_0$ on $(e_i \Pi e_i)_{2kd_i}$. 
Here we identify $\Aut_{\Pi}(I_i)^{\op}$ with $(e_i \Pi e_i)^\times_0$ via the isomorphism \eqref{tr} with $k=0$.  
Since $\Ker(f)$ is invariant under this action, it is enough to show that the orbit $\ep_i^k (e_i \Pi e_i)^\times_0$ is Zariski dense in the affine space $(e_i \Pi e_i)_{2kd_i}$.

Recall that $e_i \Pi e_i$ is graded free of finite rank as a left $\kk[\ep_i]$-module by Proposition~\ref{Prop:Piinfty}~\eqref{Prop:Piinfty:lf}.
Since the space $e_i \Pi e_i / \ep_i e_i \Pi e_i \cong e_i (\tPi / \ep_i\tPi) e_i\cong e_i\bD(\bar{I}_i)$ is non-positively graded by Proposition~\ref{Prop:dI}, we can choose a free $\kk[\ep_i]$-basis $\{ x_0, x_1, \ldots, x_m \}$ of $e_i \Pi e_i$ satisfying the followings:
\begin{enumerate}
\item $x_0 = e_i$,
\item For each $1 \le l \le m$, we have $u_l \seq \deg_1(x_l) \le 0$ and $x_l \in \mathfrak{a}$,
\item There is an integer $0 \le m' \le m$ such that we have 
\[ \begin{cases} u_l \in 2d_i \Z & \text{if $0 \le l \le m'$}, \\ 
u_l \not \in 2d_i \Z & \text{if $m' < l \le m$}.
\end{cases}
\] 
\end{enumerate}
Then we have
\[ (e_i \Pi e_i)_{2kd_i} = \kk \ep_i^k e_i \oplus \bigoplus_{l=1}^{m'} \kk \ep_i^{k - u_l/2d_i} x_l = \ep_i^k (e_i \Pi e_i)_0\]
for any $k \in \Z_{\ge 0}$.
Since each $\ep_i^{- u_l/2d_i} x_l$ with $1 \le l \le m'$ is nilpotent in $\Pi$ by Proposition~\ref{Prop:Piinfty}~\eqref{Prop:Piinfty:nilp}, we have 
\[ (e_i \Pi e_i)_0^\times = \kk^\times e_i \oplus \bigoplus_{l=1}^{m'} \kk \ep_i^{- u_l/2d_i} x_l. \]
Therefore the orbit $\ep_i^k (e_i \Pi e_i)_0^\times$ is Zariski dense in $(e_i \Pi e_i)_{2kd_i} = \ep_i^k (e_i \Pi e_i)_0$ as it is the complement of a linear subspace of codimension $1$.  
\end{proof}

\begin{Rem} \label{Rem:K_HL}
Our terminology of generic kernels coincides with the one in \cite{HL} via the isomorphism $\Phi \colon \tPi \gmod \to J_{\Gamma, W} \umod$ in Proposition~\ref{Prop:catisom} in the following sense. 
In \cite{HL}, one actually concerns the full subcategory of $J_{\Gamma, W}\umod$ consisting of modules supported on the ``semi-infinite" full subquiver $\Gamma^-$ of $\Gamma$ given by $\Gamma_0^- \seq \{ (i,p) \in I \times \Z \mid p \le -d_i\}$.  
If $p \le d_i - rh^\vee$, we can easily see from Propositions~\ref{Prop:dI} \& \ref{Prop:tc} that $e_j(q^p I_i(\infty))_u = 0$ holds unless $(j,u) \in \Gamma_0^-$, and hence $\Phi(q^{p-kd_i}K^{(i)}_{k})$ is supported on $\Gamma^-$ for any $k \in \Z_{>0}$.
By Proposition~\ref{Prop:generic}, under the condition $p \le d_i - rh^\vee$, the $J_{\Gamma, W}$-module $\Phi(q^{p-kd_i}K^{(i)}_{k})$ is identical to the generic kernel denoted by $K^{(i)}_{k, p}$ in \cite[4.3]{HL}. Note that the injective module $I_{i,p}$ in \cite[4.3]{HL} is identical to our $\Phi(q^p I_i(\infty))$ in view of Proposition~\ref{Prop:bd}.
\end{Rem}

\begin{Lem} \label{Lem:dimK}
For any $i,j \in I$ and $k \in \Z_{>0}$, we have
\begin{equation} \label{eq:dimK}
\dim_q e_jK^{(i)}_k = \frac{[kd_i]_q}{[d_i]_q} \sum_{u=0}^{rh^\vee} \tc_{ji}(u) q^u.
\end{equation}
\end{Lem}
\begin{proof}
By the definition of $K^{(i)}_k$, we have
\[ \bD(K^{(i)}_k) 
\cong q^{-kd_i} \Pi(\infty)e_i \otimes_{\kk[\ep_i]} \kk[\ep_i]/(\ep_i^k) 
\cong q^{-kd_i} \bD(\bar{I}_i) \otimes_{\kk} \kk[\ep_i]/(\ep_i^k)  \]
as graded $\kk$-vector spaces.
Therefore, 
\begin{align*}
\dim_q e_j K^{(i)}_k &= \dim_{q^{-1}} e_j\bD(K^{(i)}_k) \allowdisplaybreaks \\
&= q^{kd_i} \dim_q e_j\bar{I}_i \cdot \dim_{q^{-1}} \kk[\ep_i]/(\ep_i^k) \allowdisplaybreaks \\
&= q^{kd_i} \left(q^{-d_i} \sum_{u=0}^{rh^\vee} \tc_{ji}(u)q^u \right) \frac{1-q^{-2kd_i}}{1-q^{-2d_i}},
\end{align*}
where the last equality is due to \eqref{eq:dI}. 
This proves \eqref{eq:dimK}.
\end{proof}

\subsection{First extension groups between generic kernels}
\label{Ssec:Ext}

\begin{Prop} \label{Prop:Ext}
For each $i,j \in I$ and $k,l \in \Z_{>0}$, we have
\begin{equation} \label{eq:gextKK}
\gext^1_{\tPi}(K^{(i)}_{k}, K^{(j)}_{l}) \cong \gext^1_{\tPi}(K^{(j)}_{l}, K^{(i)}_{k}) \cong q^{-kd_i -ld_j} e_i\frac{\tPi}{\ep_i^k \tPi + \tPi \ep_j^l} e_j
\end{equation}
as graded $\kk$-vector spaces.
\end{Prop}
\begin{proof}
We only have to show the second isomorphism thanks to the opposition $\phi$ and the symmetry. 
Moreover, we know $\gext^1_{\tPi}(K^{(j)}_{l}, K^{(i)}_{k}) \cong \gext^1_{\Pi(\infty)}(K^{(j)}_{l}, K^{(i)}_{k})$ by Corollary~\ref{Cor:bd}.
Using the injective resolution \eqref{sesK} in $\Pi(\infty)\gmod$, we have
\begin{align*} 
\gext^1_{\Pi(\infty)}(K^{(j)}_{l}, K^{(i)}_{k}) &\cong \Cok\left(\ghom_{\Pi(\infty)}(K^{(j)}_{l}, q^{kd_i} I_i(\infty)) \xrightarrow{(\cdot \ep_i^k) \circ -}\ghom_{\Pi(\infty)}(K^{(j)}_{l}, q^{-kd_i}I_i(\infty)) \right) \allowdisplaybreaks \\
&\cong \Cok\left( q^{kd_i-ld_j}e_i(\tPi/\tPi\ep_j^l)e_j \xrightarrow{\ep_i^k \cdot}q^{-kd_i-ld_j}e_i(\tPi/\tPi\ep_j^l)e_j \right),
\end{align*}
where, for the second isomorphism, we used the facts $\ghom_{\Pi(\infty)}(-, I_j(\infty)) \cong e_j\bD(-)$, and $e_i\bD(K^{(j)}_{l}) \cong q^{-ld_j}e_i(\tPi/\tPi \ep_j^l)e_j$. 
This yields the desired isomorphism.
\end{proof}

To compute the dimensions of these extension groups, we need the following lemma.

\begin{Lem} \label{Lem:incl}
Let $i,j \in I$ and $k,l \in \Z_{>0}$. If $kd_i \ge ld_j$, we have 
\begin{equation} \label{eq:incl}
e_i\ep_i^k \tPi e_j \subset e_i \tPi \ep_j^l e_j
\end{equation} unless the following condition is satisfied:
\begin{itemize}
\item[$(\clubsuit)$] $\fg$ is either of type $\mathrm{C}_n$, $\mathrm{F}_4$ or $\mathrm{G}_2$, and we have $d_i = d_j = 1$, $k=l \not \in r\Z$.  \qedhere
\end{itemize}
\end{Lem}
\begin{proof}
It suffices to show that the inclusion~\eqref{eq:incl} holds in the following four cases:
\begin{enumerate}
\item \label{C1} at least one of the numbers $kd_i$ and $ld_j$ belongs to $r\Z$, and $kd_i \ge ld_j$,
\item \label{C2} $\fg$ is of type $\mathrm{C}_n$ or $\mathrm{F}_4$, $d_i=d_j=1$, $kl \not \in 2\Z$, and $k > l$,
\item \label{C3} $\fg$ is of type $\mathrm{B}_n$, $d_i = d_j = 1$, and $k \ge l$,
\item \label{C4} $\fg$ is of type $\mathrm{G}_2$, $d_i = d_j = 1$, $kl \not \in 3\Z$, and $k > l$.
\end{enumerate} 

Case \eqref{C1}: If $kd_i = mr$ for some $m \in \Z_{>0}$, we have 
\[ e_i \ep_i^k \tPi e_j = e_i \ep^m \tPi e_j = e_i \tPi \ep^m e_j = e_i \tPi \ep_j^{mr/d_j} e_j \subset e_i \tPi \ep_j^l e_j. \]
If $ld_j = mr$ for some $m \in \Z_{>0}$, we have
\[ e_i \ep_i^k \tPi e_j = e_i \ep_i^{kd_i - ld_j} \ep^m \tPi e_j = e_i \ep_i^{kd_i - ld_j} \tPi \ep^m e_j = e_i \ep_i^{kd_i - ld_j} \tPi \ep_j^l e_j \subset e_i \tPi \ep_j^l e_j.\]
 
Case \eqref{C2}: There is a positive integer $m$ such that $k > 2m > l$.
Then we have 
\[ e_i \ep_i^k \tPi e_j = e_i \ep_i^{k - 2m} \ep^m \tPi e_j = e_i \ep_i^{k - 2m} \tPi \ep^m e_j = e_i \ep_i^{k - 2m} \tPi \ep_j^{2m} e_j \subset e_i \tPi \ep_j^l e_j. \]

Case \eqref{C3}: We identify $I = \{1,\ldots,n\}$ so that $n$ is the single vertex with $d_n =1$ and we have $n-1 \sim n$. 
Then $i=j=n$ and we have to prove that $e_n \ep_n^k \tPi e_n \subset e_n \tPi \ep_n^l e_n$ under the assumption $k \ge l$.
Using the relations (R1) and (R2) at the vertices $1, \ldots, n-1$, we can easily see that the space $e_n\tPi e_n$ is spanned by the elements of the form
\[ \rho = e_n \ep_n^{k_0} \alpha_{n,n-1} \alpha_{n-1,n} \ep_n^{k_1} \alpha_{n,n-1} \alpha_{n-1,n} \ep_n^{k_2} \cdots \alpha_{n,n-1} \alpha_{n-1,n} \ep_n^{k_m} e_n\]
for some $m \in \Z_{\ge 0}$ and $k_0, \ldots, k_m \in \Z_{\ge 0}$. 
On the other hand, the relation (R2) for the vertex $n$ yields
\[ \ep_n \alpha_{n,n-1} \alpha_{n-1,n} = -\alpha_{n,n-1} \alpha_{n-1,n} \ep_n. \]
Therefore we get $\ep_n^k \rho = (-1)^{mk} \rho \ep_n^k \in e_n \tPi \ep_n^l e_n$.

Case \eqref{C4}: If there is an integer $m \in \Z_{\ge 0}$ satisfying $k > 3m > l$, we get the desired conclusion similarly as in the case \eqref{C2}.
So it is enough to consider the case $k=3m+2$ and $l=3m+1$ for some $m \in \Z_{\ge 0}$. 
We identify $I = \{1,2\}$ so that $(d_1, d_2) = (3, 1)$.
Then we have $i=j=2$ and we have to prove that $e_2 \ep_2^{3m+2} \tPi e_2 \subset e_2 \tPi \ep_2^{3m+1} e_2$. 
Since $e_2 \ep_2^{3m} = e_2 \ep^m$, it suffices to consider the case when $m=0$.
Since $\alpha_{12}\alpha_{21}=0$ by the relation (R2) for the vertex $1$, we see that the space $e_2 \tPi e_2$ is spanned by the elements of the form
\[ \rho = e_2 \ep_2^{k_0} \alpha_{21} \alpha_{12} \ep_2^{k_1} \alpha_{21} \alpha_{12} \ep_2^{k_2} \cdots \alpha_{21} \alpha_{12} \ep_2^{k_s}e_2\]
for some $s \in \Z_{\ge 0}$, $k_0, k_s \in \Z_{\ge 0}$ and $k_1, \ldots, k_{s-1} \in \Z_{> 0}$. 
On the other hand, the relation (R2) for the vertex $2$ yields
\[ \ep_2^2 \alpha_{21} \alpha_{12} = - \ep_2 \alpha_{21} \alpha_{12} \ep_2 - \alpha_{21} \alpha_{12} \ep_2^2. \]
Using this relation, we can easily show that $\ep_{2}^2 \rho \in e_2 \tPi \ep_2 e_2$ holds by induction on $s$.
This completes the proof.  
\end{proof}

\begin{Prop} \label{Prop:dext}
Let $i,j \in I$ and $k,l \in \Z_{>0}$. 
Assume $kd_i \ge ld_j$.
Then we have
\begin{equation} \label{eq:dext}
\dim_q \gext_{\tPi}^1(K^{(i)}_k, K^{(j)}_l)
= q^{-kd_i}\frac{[ld_j]_q}{[d_j]_q} \sum_{u=0}^{rh^\vee}\tc_{ij}(u)q^{-u}
\end{equation}
unless the condition $(\clubsuit)$ in Lemma~\ref{Lem:incl} is satisfied.
Even if $(\clubsuit)$ is satisfied, the difference $(\mathrm{RHS}) - (\mathrm{LHS})$ in \eqref{eq:dext} belongs to $\Z_{\ge 0}[q^{\pm 1}]$. 
\end{Prop}
\begin{proof}
By Proposition~\ref{Prop:Ext}, we have a surjection 
\begin{align*}
q^{-kd_i}\bD(e_i K^{(j)}_l) &\cong q^{-kd_i -ld_j} e_i(\tPi/\tPi \ep_j^l) e_j \allowdisplaybreaks \\
&\twoheadrightarrow
q^{-kd_i -ld_j} e_i(\tPi/(\ep_i^k \tPi + \tPi \ep_j^l)) e_j \cong 
\gext^1_{\tPi}(K^{(i)}_{k}, K^{(j)}_{l}). 
\end{align*}
Under the assumption $kd_i \ge ld_j$, it is an isomorphism unless the condition $(\clubsuit)$ is satisfied, thanks to Lemma~\ref{Lem:incl}.
Combined with Lemma~\ref{Lem:dimK}, we obtain the assertion.
\end{proof}

In Appendix~\ref{Apx}, we compute the difference $(\mathrm{RHS}) - (\mathrm{LHS})$ in \eqref{eq:dext} explicitly when the condition $(\clubsuit)$ is satisfied. See Proposition~\ref{Prop:dext2} below.

\begin{Cor}\label{cor:rigid}
For each $i \in I$ and $k \in \Z_{>0}$, the $\tPi$-module $K^{(i)}_{k}$ is rigid, that is
\[ \Ext^1_{\tPi}(K^{(i)}_{k}, K^{(i)}_{k}) = 0.\] 
\end{Cor}
\begin{proof}
Note that $\dim_\kk \Ext^1_{\tPi}(K^{(i)}_{k}, K^{(i)}_{k})$ is the constant term of $\dim_q \gext^1_{\tPi}(K^{(i)}_{k}, K^{(i)}_{k})$.
By Proposition~\ref{Prop:Ext}, it is not greater than the constant term of $q^{-kd_i}\frac{[kd_i]_{q}}{[d_i]_q}\sum_{u=0}^{rh^\vee}\tc_{ii}(u)q^{-u}$, which is zero.
\end{proof}

\subsection{Pole orders of normalized $R$-matrices}
\label{Ssec:conj}
Let $U'_q(\widehat{\fg})$ be the (untwisted) quantum affine algebra associated with our simple Lie algebra $\fg$.
This is a Hopf algebra defined over an algebraic closure $\kb \seq \ol{\Q(q)}$ of the field of rational functions in $q$.
The category $\CC$ of finite-dimensional $U'_q(\widehat{\fg})$-modules (of type $\mathbf{1}$) forms an interesting $\kb$-linear monoidal abelian category.
Around $U'_q(\widehat{\fg})$, we basically follow the convention of \cite[\S 5 and \S 6]{FO} except that we replace the quantum parameter $q$ therein with $q^r$ (and hence $q_s$ therein is $q$ here). 

It is well-known (see \cite[Chapter 12]{CP} for example) that simple $U'_q(\widehat{\fg})$-modules in $\CC$ are parametrized by the set $(1+x\kb[x])^I$ of $I$-tuples of polynomials with constant terms $1$, which are called \emph{Drinfeld polynomials}. 
For a nonzero scalar $a \in \kb^\times$ and a module $M \in \CC$, we can define another module $M_a \in \CC$, called a spectral parameter shift of $M$, by twisting its module structure with an automorphism of the Hopf algebra $U'_q(\widehat{\fg})$. 
If $L$ is a simple module of $\CC$ associated with Drinfeld polynomials $\pi(x) \in (1+x\kb[x])^I$, its spectral parameter shift $L_a$ is associated with $\pi(ax)$. 
It defines an action of the group $\kb^\times$ on the monoidal category $\CC$.

For each $i \in I$ and $k \in \Z_{>0}$, we denote by $V^{(i)}_k$ the simple $U'_q(\widehat{\fg})$-module in $\CC$ associated with the Drinfeld polynomials $\pi^{(i)}_k(x) = (\pi^{(i)}_{k,j}(x))_{j \in I}$ given by 
\[ 
\pi^{(i)}_{k,j}(x) \seq \begin{cases}
(1-q^{(k-1)d_i}x)(1-q^{(k-3)d_i}x) \cdots (1-q^{-(k-1)d_i}x) & \text{if $j=i$}, \\
1 & \text{if $j \neq i$}.
\end{cases}
\]
These modules $\{ V^{(i)}_k \}_{i \in I, k \in \Z_{>0}}$ and their spectral parameter shifts are called the \emph{Kirillov-Reshetikhin modules} (KR modules for short). 

\begin{Rem} \label{Rem:qch}
When $\kk = \C$, Hernandez-Leclerc's {geometric character formula} \cite[Theorem 4.8]{HL} tells us that the $F$-polynomial of the $J_{\Gamma, W}$-module $\Phi(q^pK^{(i)}_k)$ in the sense of \cite{DWZ2} gives the $q$-character of the KR module $V^{(i)}_{k, q^p}$ in the sense of \cite{FR}. 
To explain it more precisely, we need some notation.
Recall that each $M \in \CC$ has a spectral decomposition $M = \bigoplus_{\gamma \in \Gg} M_\gamma$ with respect to the action of the quantum analog of loop Cartan part, whose spectra are parametrized by the Grothendieck group $\Gg$ of the multiplicative monoid $(1+x\kb[x])^I$.
Then the $q$-character of $M$ is the element $\chi_q(M) \seq \sum_{\gamma \in \Gg} (\dim_{\kb} M_\gamma)\gamma$ of the group ring $\Z \Gg$.  
For each $i \in I$ and $p \in \Z$, we set $Y_{i,p}^{\pm 1} \seq \pi^{(i)}_{1}(q^px)^{\pm 1} \in \Gg$ and
\[ A_{i,p} \seq Y_{i,p-d_i}Y_{i,p+d_i}
\prod_{j\colon c_{ji} = -1}Y_{j,p}^{-1}
\prod_{j\colon c_{ji} = -2}Y_{j,p-1}^{-1}Y_{j,p+1}^{-1}
\prod_{j\colon c_{ji} = -3}Y_{j,p-2}^{-1}Y_{j,p}^{-1}Y_{j,p+2}^{-1}.\]
Note that the transformation from $Y_{i,p}$'s to $A_{i,p}$'s is given by the quantum Cartan matrix $C(q)$.  
Then the geometric character formula is expressed as
\begin{equation} \label{eq:gcf} 
\chi_q(V^{(i)}_{k, q^p}) = \pi^{(i)}_{k}(q^px)\sum_{\nu \in (\Z_{\ge 0})^{I \times \Z}} \chi\left(\Gr_\nu(\Phi(q^pK^{(i)}_{k}))\right) \prod_{(j,s) \in I \times \Z} A_{j,s}^{-\nu_{j,s}},
\end{equation}
where $\Gr_\nu(K)$ is the projective variety parametrizing the $J_{\Gamma, W}$-submodules $L \subset K$ satisfying $\dim_\C e_{j,s}L = \nu_{j,s}$ for all $(j,s) \in I \times \Z$, and $\chi(\Gr_\nu(K))$ is its Euler characteristic.

As explained in \cite[Remark 4.14]{HL}, if we consider the case $(k,p)=(1,0)$ and look at the term with $\nu_{j,s} = \dim_\C e_{j,s}\Phi(K^{(i)}_1)=\dim_\C (e_j \bar{I}_{i})_{s-d_i}$ in \eqref{eq:gcf}, we obtain
\[ Y_{i,0}Y_{i^*,rh^\vee} = \prod_{(j,s) \in I \times \Z}A_{i,s+d_i}^{\dim_\C (e_j \bar{I}_{i})_{s}} \]
by \cite[Lemma 6.8]{FM}. 
Taking \cite[Lemma 6.13]{FM} (or Corollary~\ref{Cor:dimIbd}) into account, it gives an alternative proof of Proposition~\ref{Prop:dI} when $\kk = \C$.
By the similar argument, Lemma~\ref{Lem:dimK} can also be deduced from \eqref{eq:gcf}.
\end{Rem}

Let $z$ be an indeterminate and $N_z = N \otimes_\kk \kk(z)$ the formal spectral parameter shift of a module $N \in \CC$.
For any pair $(M,N)$ of simple modules in $\CC$, we can associate the \emph{normalized $R$-matrix}
as a $U'_q(\widehat{\fg})\otimes_\kk \kk(z)$-isomorphism
\[ R_{M,N} \colon M \otimes N_z \to N_z \otimes M \]
which sends a specific cyclic vector of $M\otimes N_z$ to that of $N_z \otimes M$. 
Since $R_{M,N}$ can be seen as a matrix-valued rational function in $z$, one can consider its denominator $d_{M,N}(z) \in \kb[z]$, which is uniquely determined as a monic polynomial in the variable $z$ with $d_{M,N}(0) \neq 0$. 
For any nonzero scalars $a, b \in \kb^\times$, we have
\begin{equation} \label{spshiftd}
d_{M_a, N_b}(z) \equiv d_{M,N}((b/a)z) \pmod {\kb^\times}.
\end{equation}
The denominator $d_{M,N}(z)$ contains some important information of the structure of the tensor product module $M \otimes N$ (see~\cite{KKKO} for example).  
For the denominators between the KR modules, we have the following conjectural formula in terms of the matrix $\tC(q)$. 

\begin{Conj}[{cf.~\cite[Conjecture 6.7]{FO}}] \label{Conj:KRd}
Let $i,j \in I$ and $k,l \in \Z_{>0}$. 
Assume $kd_i \ge ld_j$.
Then we have
\begin{equation} \label{eq:KRd}
d_{V^{(i)}_k, V^{(j)}_l}(z) = d_{V^{(j)}_l, V^{(i)}_k}(z) 
= \prod_{a=0}^{l-1}\prod_{u=0}^{rh^\vee}\left( z-q^{u+kd_i+(2a-l+1)d_j}\right)^{\tc_{ij}(u)}.
\end{equation}
unless the condition $(\clubsuit)$ in Lemma~\ref{Lem:incl} is satisfied.
\end{Conj}

Conjecture~\ref{Conj:KRd} is known to be true in the following cases:
\begin{itemize}
\item $\fg$ is either of type $\mathrm{A}_n, \mathrm{B}_n, \mathrm{C}_n, \mathrm{D}_n$ or $\mathrm{G}_2$ (cf.~\cite[Theorem 6.9]{FO}),
\item $\fg$ is of any type and $(k,l) = (r/d_i, 1)$ (cf.~\cite[Proposition 6.5]{FO}).
\end{itemize}

\begin{Rem}
Strictly speaking, the statement of \cite[Conjecture 6.7]{FO} is slightly different from that of Conjecture~\ref{Conj:KRd} above.
Actually, it was conjectured in \cite[Conjecture 6.7]{FO} that the equality~\eqref{eq:KRd} should hold also when
\begin{itemize}
\item[$(\ast)$] $\fg$ is of type $\mathrm{CFG}$, and we have $d_i = d_j = 1, k = l \in \Z_{\ge r} \setminus r\Z_{>0}$.
\end{itemize}
Then it was claimed in \cite[Theorem 6.9]{FO} that the equality~\eqref{eq:KRd} is indeed true for type $\mathrm{C}_n$ and $\mathrm{G}_2$ under the condition $(\ast)$ based on the results of~\cite{OS}.
However, after the publication of \cite{FO}, it has turned out that there are some gaps in the proofs in \cite{OS} for the case $(\ast)$ (the authors thank Se-jin Oh for clarifying this point).
Thus we have excluded the case $(\ast)$ in the statement of Conjecture~\ref{Conj:KRd}.  
\end{Rem}

\begin{Def}\hfill
\begin{enumerate}
\item
For $a \in \kb$ and $f(z) \in \kb[z]$, let $\zero_{z=a} f(z)$ denote the order of zeros of $f(z)$ at $z=a$.
For $f(z) \in \kb[z]$ with $f(0) \neq 0$, we define its divisor $\Div f(z)$ to be an element of the group ring $\Z[\kb^\times]$ given by $\Div f(z) \seq \sum_{a \in \kb^\times} (\zero_{z=a} f(z)) a$.
\item For a pair $(M,N)$ of simple modules in $\CC$, we set
\[ \mathfrak{o}(M,N) \seq \zero_{z=1} d_{M,N}(z), \qquad \mathfrak{O}(M,N) \seq \Div d_{M,N}(z). \]
Note that $\mathfrak{o}(M,N)$ is the coefficient of $1 \in \kb^\times$ in $\mathfrak{O}(M,N)$. \qedhere
\end{enumerate}
\end{Def}

For any $i,j \in I$ and $k, l \in \Z_{>0}$, all the zeros of $d_{V^{(i)}_k, V^{(j)}_l}(z)$ belong to $q^\Z \subset \kb^\times$. This follows from \cite[Propositions 2.11 \& 2.12]{KKOP} and the known formulas of denominators and universal coefficients for the fundamental modules.
In particular, $\mathfrak{O}(V^{(i)}_k, V^{(j)}_l)$ can be thought of a Laurent polynomial in $q$.

\begin{Lem} \label{Lem:KRO}
The equation \eqref{eq:KRd} is equivalent to the equation 
\begin{equation} \label{eq:KRO}
\mathfrak{O}(V^{(i)}_k, V^{(j)}_l)= \mathfrak{O}(V^{(j)}_l, V^{(i)}_k) = q^{kd_i} \frac{[ld_j]_q}{[d_j]_q}\sum_{u=0}^{rh^\vee} \tc_{ij}(u)q^u.
\end{equation}
\end{Lem}
\begin{proof}
Straightforward.
\end{proof}

Thanks to Proposition~\ref{Prop:Ext}, Proposition~\ref{Prop:dext} and Lemma~\ref{Lem:KRO}, the following conjecture is equivalent to Conjecture~\ref{Conj:KRd}.

\begin{Conj}[$\Leftrightarrow$ Conjecture~\ref{Conj:KRd}]
Let $i,j \in I$ and $k, l \in \Z_{>0}$. Assuming $kd_i \ge ld_j$, we have
\begin{equation} \label{eq:O=dext}
\mathfrak{O}(V^{(i)}_k, V^{(j)}_l) = \dim_{q^{-1}} \gext_{\tPi}^1(K^{(i)}_k, K^{(j)}_l) 
\end{equation}
unless the condition $(\clubsuit)$ in Lemma~\ref{Lem:incl} is satisfied.
\end{Conj}

Moreover, we find that the equality \eqref{eq:O=dext} still holds in some cases when the condition $(\clubsuit)$ is satisfied as in the next proposition, whose proof is given later in Appendix~\ref{Apx}.  

\begin{Prop} \label{Prop:evid}
The equality \eqref{eq:O=dext} still holds even if $\fg$ is either of type $\mathrm{C}_n$, $\mathrm{F}_4$ or $\mathrm{G}_2$, and we have $d_i=d_j=1, k=l<r$.
\end{Prop}

This motivates us to propose the following conjecture, which generalizes Conjecture~\ref{Conj:KRd}. 

\begin{Conj} \label{Conj:main1}
For any $i,j \in I$ and $k,l \in \Z_{>0}$, the equality~\eqref{eq:O=dext} holds.
\end{Conj}

\begin{Lem}
The equation~\eqref{eq:O=dext} holds if and only if the equation 
\begin{equation} \label{eq:o=dExt}
\mathfrak{o}(V^{(i)}_{k,q^p}, V^{(j)}_{l, q^s}) = \dim_\kk \Ext_{\tPi}^1(q^pK^{(i)}_k, q^sK^{(j)}_l)
\end{equation}
holds for any $p, s \in \Z$.
\end{Lem}
\begin{proof}
Note that the left hand side of \eqref{eq:o=dExt} is the constant term of 
\[ \mathfrak{O}(V^{(i)}_{k, q^p}, V^{(j)}_{l, q^s}) = q^{p-s} \mathfrak{O}(V^{(i)}_{k}, V^{(j)}_{l}),  \]
where the equality follows from the property \eqref{spshiftd}.
On the other hand, the right hand side of \eqref{eq:o=dExt} is the constant term of
\[ \dim_{q} \gext_{\tPi}^1(q^pK^{(i)}_k, q^sK^{(j)}_l) = q^{s-p} \dim_{q} \gext_{\tPi}^1(K^{(i)}_k, K^{(j)}_l). \]
From these observations, we obtain the assertion.
\end{proof}

Thus, we can rephrase Conjecture~\ref{Conj:main1} as follows.  

\begin{Conj}[$\Leftrightarrow$ Conjecture~\ref{Conj:main1}] \label{Conj:main2}
For any $i,j \in I$, $k,l \in \Z_{>0}$ and $p,s \in \Z$, the equality~\eqref{eq:o=dExt} holds. 
\end{Conj}

\begin{Rem}
We can formally extend Conjecture~\ref{Conj:main1} (or \ref{Conj:main2}) beyond the KR modules. 
In the context of monoidal categorification of cluster algebras, for each real simple module $M \in \CC$ corresponding to a cluster monomial, one can define the corresponding generic kernel $K_M \in \tPi \gmod$ as discussed in \cite[\S5.2.2]{HL}. (Note that Conjecture~5.3 therein is now a theorem thanks to the recent progress~\cite{KKOP2}.) 
Then, as a generalization of the conjectural equality~\eqref{eq:o=dExt}, we may imagine that the equality 
\begin{equation}\label{eq:q=dExtgeneral}
\mathfrak{o}(M, N) = \dim_\kk \Ext_{\tPi}^1(K_M, K_N)
\end{equation}  
holds for any two simple modules $M,N \in \CC$ corresponding to cluster monomials. 
However, we do not have any pieces of evidence for the validity of such an equality \eqref{eq:q=dExtgeneral} except for the KR modules at this moment. 
\end{Rem}

\begin{appendix}
\section{Computations in the exceptional case $(\clubsuit)$}
\label{Apx}
In this appendix, we assume that our Lie algebra $\fg$ is of type $\mathrm{C}_n$, $\mathrm{F}_4$ or $\mathrm{G}_2$.
Following the convention of \cite[Chapter 4]{Kac}, we identify $I$ with the set $\{1,2,\ldots,n\}$ so that we have $i \sim j$ if and only if $|i-j|=1$, and
\[(d_1, \ldots, d_n) = \begin{cases}
(1,\ldots,1.2) & \text{if $\fg$ is of type $\mathrm{C}_n$}, \\
(2,2,1,1) & \text{if $\fg$ is of type $\mathrm{F}_4$}, \\
(3,1) & \text{if $\fg$ is of type $\mathrm{G}_2$}.
\end{cases}\] 

\begin{Prop} \label{Prop:dext2}
Let $\fg$ be either of type $\mathrm{C}_n$, $\mathrm{F}_4$ or $\mathrm{G}_2$.
Suppose that $i, j \in I$ and $k \in \Z_{> 0}$ satisfy $d_i = d_j =1$ and $k \not \in r\Z$.
Then we have
\begin{equation} \label{eq:dext2}
\dim_q \gext_{\tPi}^1(K^{(i)}_{k}, K^{(j)}_{k}) = q^{-k}[k]_q\sum_{u=0}^{rh^\vee}\tc_{ij}(u)q^{-u} - \Delta_{ij}(q^{-1}),
\end{equation}
where $\Delta_{ij}(q) \seq q^2 \dim_{q^{-1}}e_i( (\ep_i \tPi + \tPi \ep_j)/\tPi \ep_j) e_j $, which is explicitly computed as follows:
\begin{enumerate}
\item \label{Prop:dext2:C} 
When $\fg$ is of type $\mathrm{C}_n$ and $1 \le i,j <n$, we have
\[ \Delta_{ij}(q) = \sum_{a=1}^{i+j-n}q^{2n-i-j+2a+2}.\]
Note that $\Delta_{ij}(q) =0$ if $i+j \le n$.
\item \label{Prop:dext2:F}
When $\fg$ is of type $\mathrm{F}_4$, we have
\[ \Delta_{33}(q) = q^4+q^8+q^{10}+q^{14}, \quad \Delta_{34}(q)=\Delta_{43}(q)=q^9, \quad \Delta_{44}(q)=0. \] 
\item \label{Prop:dext2:G}
When $\fg$ is of type $\mathrm{G}_2$, we have
$ \Delta_{22}(q) = q^6.$ \qedhere
\end{enumerate}
\end{Prop}

We give a proof of Proposition~\ref{Prop:dext2} in \S\ref{Ssec:pfdext2} below. 

\subsection{Proof of Proposition~\ref{Prop:evid}} \label{Ssec:list}
By Proposition~\ref{Prop:dext2}, it is enough to show that the equality
\begin{equation}
\mathfrak{O}(V^{(i)}_{k}, V^{(j)}_k) = \frac{1-q^{2k}}{1-q^2}\sum_{u=0}^{rh^\vee} \tc_{ij}(u)q^{u+1} - \Delta_{ij}(q)
\end{equation}
holds under the assumption $d_i = d_j = 1$ and $0 < k < r$.
We can check it directly using the known formulas listed below.
They are quoted from \cite[\S 4.3 and Appendix A]{FO}.
Note that the denominators are originally computed in \cite{AK,OS19}.

Type $\mathrm{C}_n$: If $1 \le i, j < n$, we have
\begin{align*} 
\sum_{u=0}^{rh^\vee} \tc_{ij}(u)q^u &= \sum_{u=1}^{\min(i,j)} (q^{|i-j|+2u-1} + q^{2n-i-j+2u+1}), \allowdisplaybreaks \\
\mathfrak{O}(V^{(i)}_1, V^{(j)}_1) &=  \sum_{u=1}^{\min(i,j,n-i,n-j)} q^{|i-j|+2u} + \sum_{u=1}^{\min(i,j)}q^{2n-i-j+2u+2}.
\end{align*}

Type $\mathrm{F}_4$: We have
\begin{align*}
\sum_{u=0}^{rh^\vee} \tc_{ij}(u)q^u &= \begin{cases}
q+q^3+q^5+2q^7+2q^9+2q^{11}+q^{13}+q^{15}+q^{17} &\text{if $(i,j) = (3,3)$}, \\
q^2+q^6+q^8+q^{10}+q^{12}+q^{16} &\text{if $\{i,j\}=\{3,4\}$}, \\
q+q^7+q^{11}+q^{17} &\text{if $(i,j)=(4,4)$},
\end{cases} \allowdisplaybreaks\\
\mathfrak{O}(V^{(i)}_1, V^{(j)}_1) &= \begin{cases}
q^2+q^6+q^8+q^{10}+2q^{12}+q^{16}+q^{18} &\text{if $(i,j) = (3,3)$}, \\
q^3+q^7+q^{11}+q^{13}+q^{17} &\text{if $\{i,j\}=\{3,4\}$}, \\
q^2+q^8+q^{12}+q^{18} &\text{if $(i,j)=(4,4)$}.
\end{cases}
\end{align*}

Type $\mathrm{G}_2$: We have
\begin{align*}
\sum_{u=0}^{rh^\vee} \tc_{22}(u)q^u &= q+q^5+q^7+q^{11},\allowdisplaybreaks\\
\mathfrak{O}(V^{(2)}_1, V^{(2)}_1) &= q^2+q^8+q^{12}, \allowdisplaybreaks\\
\mathfrak{O}(V^{(2)}_2, V^{(2)}_2) &= q^2+q^4+2q^8+q^{10}+q^{12}+q^{14}.
\end{align*}

\subsection{Proof of Proposition~\ref{Prop:dext2}}
\label{Ssec:pfdext2}
We shall prove the equality~\eqref{eq:dext2} by applying $\dim_q (-)$ to
the short exact sequence
\[ 0 \to q^{-2k}e_i\frac{\ep_i^k \tPi + \tPi \ep_j^k}{\tPi \ep_j^k}e_j \to q^{-2k}e_i\frac{\tPi}{\tPi \ep_j^k}e_j \to q^{-2k}e_i \frac{\tPi}{\ep_i^k \tPi + \tPi \ep_j^k}e_j \to 0\]
and then by using Lemma~\ref{Lem:dimK} and Proposition~\ref{Prop:Ext}.
To this end, we need to show that there is an isomorphism 
\[q^{-2k}e_i\frac{\ep_i^k \tPi + \tPi \ep_j^k}{\tPi \ep_j^k}e_j \cong q^{-2}e_i\frac{\ep_i \tPi + \tPi \ep_j}{\tPi \ep_j}e_j \]
of graded $\kk$-vector spaces.
Writing $k=mr+s$ with $m \in \Z_{\ge 0}$ and $0<s<r$, we have
\[ 
q^{-2k}e_i\frac{\ep_i^k \tPi + \tPi \ep_j^k}{\tPi \ep_j^k}e_j \cong q^{-2k}e_i\frac{(\ep_i^s \tPi + \tPi \ep_j^s)\ep^m}{\tPi \ep_j^s \ep^m}e_j \cong q^{-2s}e_i\frac{\ep_i^s \tPi + \tPi \ep_j^s}{\tPi \ep_j^s}e_j.
\]
Therefore, when $s=1$, we are done.
When $s=2$, it is necessarily the case when $\fg$ is of type $\mathrm{G}_2$ and $i=j=2$.
Then, by the similar argument as the proof of Lemma~\ref{Lem:incl} (Case \eqref{C4}), we see that $e_2 \ep_2^2 \tPi e_2 = e_2 \ep_2 \tPi \ep_2 e_2$ and hence
\[ q^{-4}e_2\frac{\ep_2^2 \tPi + \tPi \ep_2^2}{\tPi \ep_2^2}e_2 = q^{-4}e_2\frac{(\ep_2 \tPi + \tPi \ep_2)\ep_2}{\tPi \ep_2^2}e_2 \cong q^{-2} e_2\frac{\ep_2 \tPi + \tPi \ep_2}{\tPi \ep_2}e_2.\]  
This completes the proof of the equality \eqref{eq:dext2}.

Next, we shall compute $\Delta_{ij}(q)$ explicitly.
By the symmetry, we may assume $1 \le i \le j \le n$.
Note that the graded vector space $q^{-2}e_i( (\ep_i \tPi + \tPi \ep_j)/\tPi \ep_j)e_j$ is the image of the left multiplication map
\[ \ep_i \cdot \colon e_i(\tPi/\tPi\ep_j)e_j \to q^{-2} e_i(\tPi/\tPi\ep_j)e_j.\]
We shall compute this map case-by-case by choosing an explicit $\kk$-basis of the space $e_i(\tPi/\tPi\ep_j)e_j$ with the help of Proposition~\ref{Prop:dI} and the explicit formulas of $\tc_{ij}(u)$ listed in \S\ref{Ssec:list}.
To simplify the notation, for any $a,b \in I = \{1,\ldots,n\}$, we set
\[\alpha_{ab} \seq \begin{cases}
\alpha_{a, a+1} \alpha_{a+1,a+2} \cdots \alpha_{b-1,b} & \text{if $a < b$}, \\
\alpha_{a,a-1} \alpha_{a-1,a-2} \cdots \alpha_{b+1,b} & \text{if $a > b$}, \\
e_a & \text{if $a = b$}.
\end{cases} \] 

Type $\mathrm{C}_n$: 
With $1 \le i \le j <n$ fixed, we set
\[
\rho_a \seq \alpha_{ia}\alpha_{aj}, \qquad \rho'_a \seq \alpha_{ia} \alpha_{an} \alpha_{nj}
\]
for each $1 \le a \le i$, and regard them as elements of $e_i (\tPi / \tPi \ep_j) e_j$.
Note that they are homogeneous with $\deg_1 (\rho_a) = 2a -i-j$ and $\deg_1 (\rho'_a) = 2a-2n-i+j-2$.
Thanks to Proposition~\ref{Prop:dI} and the above explicit formula of $\tc_{ij}(u)$, we can easily see that the set $\{ \rho_a, \rho'_a\}_{1\le a \le i}$ gives a $\kk$-basis of $e_i (\tPi / \tPi \ep_j) e_j$.
Using the relations (R1) and (R2), we compute $\ep_i \rho_a =0$ for any $1 \le a \le i$, and
\begin{align*} 
\ep_i \rho'_a &= \pm \alpha_{ia} \alpha_{a,n-1} \alpha_{n-1,n-2} \alpha_{n-2,n-1} \alpha_{n-1,j}  \\
&=\begin{cases}
0 & \text{if $1 \le a \le \min(i, n-j) $}, \\
\rho_{a-n+j} & \text{if $n-j< a \le i$}. 
\end{cases}
\end{align*}
Therefore, we obtain \[\Delta_{ij}(q) = \sum_{n-j<a\le i}q^{-\deg_1(\rho'_{a})} 
= \sum_{a=1}^{i+j-n} q^{2n-i-j+2a+2}.\]

Type $\mathrm{F}_4$ and $\mathrm{G}_2$:
When $(i,j) = (4,4)$ in type $\mathrm{F}_4$, the left multiplication map by $\ep_4$ from $e_4(\tPi/\tPi\ep_4)e_4$ to $q^{-2} e_4(\tPi/\tPi\ep_4)e_4$ is zero because $\dim_{q} e_4(\tPi/\tPi\ep_4)e_4 = q^{-1} + q^{-7} + q^{-11} + q^{-17}$. 
Therefore we have $\Delta_{44}(q)=0$.
For the other case, we compute a homogeneous $\kk$-basis of the space $e_i(\tPi/\tPi\ep_j)e_j$ and its image under the map $x \mapsto \ep_i x$ as in the Tables \ref{Table:F1}, \ref{Table:F2} and $\ref{Table:G}$ below. 
These computations yield the desired formulas of $\Delta_{ij}(q)$.

\begin{table}[H]
\[
\begin{array}{|c|ccc|} \hline
\deg_1(x)& \text{basis element $x$} & \mapsto & \ep_i x \\ \hline \hline
0 & e_3 & \mapsto & 0 \\ \hline
-2 & \alpha_{34}\alpha_{43} & \mapsto & 0 \\ \hline
-4 & \alpha_{32}\alpha_{23} & \mapsto & \pm \alpha_{34}\alpha_{43} \\ \hline
\multirow{2}{*}{$-6$} & \alpha_{32}\alpha_{24}\alpha_{43} & \mapsto & 0 \\ 
& \alpha_{34}\alpha_{42}\alpha_{23} & \mapsto & 0 \\ \hline
\multirow{2}{*}{$-8$} & \alpha_{31}\alpha_{13} & \mapsto & \pm \alpha_{32}\alpha_{24}\alpha_{43} \pm \alpha_{34}\alpha_{42}\alpha_{23} \\ 
& \alpha_{34}\alpha_{42}\alpha_{24}\alpha_{43} & \mapsto & 0 \\ \hline
\multirow{2}{*}{$-10$} & \alpha_{31}\alpha_{14}\alpha_{43} & \mapsto & \pm\alpha_{34}\alpha_{42}\alpha_{24}\alpha_{43} \\
& \alpha_{34}\alpha_{41}\alpha_{13} & \mapsto & \pm\alpha_{34}\alpha_{42}\alpha_{24}\alpha_{43} \\ \hline
-12 & \alpha_{34}\alpha_{41}\alpha_{14}\alpha_{43} & \mapsto & 0 \\ \hline
-14 & \alpha_{31}\alpha_{13}\ep_3\alpha_{31}\alpha_{13} & \mapsto & \pm \alpha_{34}\alpha_{41}\alpha_{14}\alpha_{43} \\ \hline
-16 & \alpha_{32}\alpha_{24}\alpha_{42}\alpha_{24}\alpha_{42}\alpha_{23} & \mapsto & 0 \\ \hline
\end{array}
\]
\caption{$(i,j)=(3,3)$ in type $\mathrm{F}_4$}
\label{Table:F1}
\end{table}

\begin{table}[H]
\[\begin{array}{|c|ccc|} \hline
\deg_1(x) & \text{basis element $x$} & \mapsto & \ep_i x \\
\hline \hline
-1& \alpha_{34} & \mapsto & 0 \\ \hline
-5& \alpha_{32}\alpha_{24}& \mapsto &0  \\ \hline 
-7&\alpha_{34}\alpha_{42}\alpha_{24}&\mapsto &0 \\ \hline 
-9&\alpha_{31}\alpha_{14}&\mapsto&\pm\alpha_{34}\alpha_{42}\alpha_{24} \\ \hline 
-11&\alpha_{34}\alpha_{41}\alpha_{14}&\mapsto&0 \\ \hline  
-15&\alpha_{31}\alpha_{13}\ep_3 \alpha_{31}\alpha_{14}&\mapsto&0 \\ \hline  
\end{array}\]
\caption{$(i,j)=(3,4)$ in type $\mathrm{F}_4$}
\label{Table:F2}
\end{table}

\begin{table}[H]
\[
\begin{array}{|c|ccc|} \hline
\deg_1(x) & \text{basis element $x$} & \mapsto & \ep_i x \\ \hline \hline
0 & e_2 & \mapsto & 0 \\ \hline
-4 & \ep_2\alpha_{21}\alpha_{12} & \mapsto & 0 \\ \hline
-6 & \alpha_{21}\alpha_{12} & \mapsto & \ep_2\alpha_{21}\alpha_{12} \\ \hline
-10 & \alpha_{21}\alpha_{12}\ep_2\alpha_{21}\alpha_{12} & \mapsto & 0 \\ \hline
\end{array}
\]
\caption{$(i,j)=(2,2)$ in type $\mathrm{G}_2$}
\label{Table:G}
\end{table}
\end{appendix}
\subsection*{Acknowledgments}
We are grateful to David Hernandez, Osamu Iyama, Geoffrey Janssens, Syu Kato, Bernhard Keller, Taro Kimura, Bernard Leclerc, Yuya Mizuno, Se-jin Oh and Hironori Oya for useful discussions or comments.
Moreover, the idea of expressing deformed Cartan matrices in terms of the graded Euler-Poincar\'{e} pairings was inspired in part by Bernhard Keller's talk \cite{Kels}.
We thank him for sharing his insights.
Yuya Mizuno asked us about Jordan-H\"older filtrations of projective modules over the generalized preprojective algebras. The idea for Corollary~\ref{Cor:filtJ} was motivated from his question, and we thank him.
A part of this project was motivated by \cite{HK}, on which the second named author is assigned to give his lecture at Winter School ``Koszul Algebra and Koszul Duality" at Osaka City University. 
He thanks the organizers of the school for giving this opportunity.
\bibliography{ref}
\end{document}